\documentclass[journal]{IEEEtran}
\usepackage{hyperref}
\hypersetup{colorlinks=true,linkcolor=blue,citecolor=blue}
\usepackage{times}
\usepackage{graphicx,graphics,epsf,epstopdf}
\usepackage{comment}
\usepackage{amsmath,amsxtra,amsfonts,amscd,amssymb,bm}
\usepackage{amsthm}
\allowdisplaybreaks % automatic page breaks in math environments
\usepackage{supertabular,multirow}
\usepackage{booktabs}
\usepackage{scrextend} % only for \footref

\usepackage{subcaption}
\usepackage{lineno}
\usepackage{multirow,bm}
\usepackage{color}
\usepackage{algorithm}
\usepackage{algorithmic}
\usepackage{booktabs}
\usepackage{rotating}
\usepackage{enumitem}
\usepackage{lineno}
\usepackage{cite}
\usepackage{url}
\usepackage{pdflscape}
\usepackage{makecell}
\usepackage{multicol}
\usepackage{tikz}
\usepackage{tcolorbox}

\numberwithin{equation}{section}

 \newtheorem{theorem}{Theorem}
\newtheorem{lemma}{Lemma}
\newtheorem{assumption}{Assumption}
\newtheorem{proposition}{Proposition}

\newtheorem{definition}{Definition}

\newcommand{\ba}{\begin{array}}
\newcommand{\ea}{\end{array}}

\newcommand{\bit}{\begin{itemize}}
\newcommand{\eit}{\end{itemize}}
\newcommand{\be}{\begin{equation}}
\newcommand{\ee}{\end{equation}}
\newcommand{\bea}{\begin{eqnarray}}
\newcommand{\eea}{\end{eqnarray}}

\newcommand{\st}{\mathrm{s.t.}}

\newcommand{\prox}{\mathrm{prox}}
\newcommand{\argmin}{\mathop{\mathrm{arg\,min}}}

\newcommand{\by}{\mathrm{\bf y}}
\newcommand{\bs}{\mathrm{\bf s}}
\newcommand{\bd}{\mathrm{\bf d}}

\newcommand{\bx}{\mathbf{x}}
\newcommand{\bw}{\mathbf{w}}
\newcommand{\bz}{\mathbf{z}}
\newcommand{\bW}{\mathbf{W}}

\newcommand{\Rmn}[1]{\uppercase\expandafter{\romannumeral#1}}
\renewcommand{\theenumi}{\roman{enumi}}

\newcommand{\llangle}{\left \langle}
\newcommand{\rrangle}{\right \rangle}

\newcommand{\Pcal}{\mathcal{P}}

\numberwithin{equation}{section}

\newcommand{\Mcal}{\mathcal{M}}

\newcommand{\sumj}{\sum_{j=1}^n}
\newcommand{\wt}{W_{ij}^t}
\newcommand{\grad}{\mathrm{grad}}

\newcommand{\R}{\mathbb{R}}

\numberwithin{theorem}{section}
\newcommand{\iprod}[2]{\left \langle #1, #2 \right \rangle }

\newcommand{\tr}{\mathrm{tr}}

\begin{document}

\title{Decentralized Douglas-Rachford splitting methods for smooth optimization over compact submanifolds}

\author{Kangkang Deng, Jiang Hu, Hongxia Wang
	\thanks{Kangkang Deng is with the Department of Mathematics, National University of Defense Technology, China. 
 Email: ~freedeng1208@gmail.com}
	\thanks{Jiang Hu is with Massachusetts General Hospital and Harvard Medical School, Harvard University, Boston, MA 02114. Email:~hujiangopt@gmail.com. }%(Corresponding author: Jiang Hu).
 \thanks{Hongxia Wang is with the Department of Mathematics, National University of Defense Technology, China. 
 Email: ~wanghongxia@nudt.edu.cn}
 \thanks{{\it Corresponding author: Jiang Hu.}}
}

\maketitle

\begin{abstract}
We study decentralized smooth optimization problems over compact submanifolds. Recasting it as a composite optimization problem, we propose a decentralized Douglas-Rachford splitting algorithm, DDRS. When the proximal operator of the local loss function does not have a closed-form solution, an inexact version of DDRS, iDDRS, is also presented.
Both algorithms rely on an ingenious integration of the nonconvex Douglas-Rachford splitting algorithm with gradient tracking and manifold optimization. We show that our DDRS and iDDRS achieve the best-known convergence rate of $\mathcal{O}(1/K)$. The main challenge in the proof is how to handle the nonconvexity of the manifold constraint. To address this issue, we utilize the concept of proximal smoothness for compact submanifolds. This ensures that the projection onto the submanifold exhibits convexity-like properties, which allows us to control the consensus error across agents.
Numerical experiments on the principal component analysis are conducted to demonstrate the effectiveness of our decentralized DRS compared with the state-of-the-art ones.
\end{abstract}


\section{Introduction}

Owing to concerns about privacy and robustness, decentralized optimization over manifolds has garnered significant attention in machine learning, optimization control, and signal processing. Examples include principal component analysis \cite{scaglione2008decentralized,ye2021deepca,chen2021decentralized}, low-rank matrix completion \cite{boumal2015low,mishra2019riemannian,hu2023decentralized}, and low-dimension subspace learning \cite{mishra2019riemannian,hu2023decentralized}.
The problem can be mathematically formulated as follows:
\be \label{prob:original}  
\begin{aligned}
   \min_{x_1,\cdots,x_n} \quad & \sum_{i=1}^n f_i(x_i) \\
   \st \quad & x_1 = \cdots = x_n, \; x_i \in \Mcal, \; \forall i=1,\cdots,n,
\end{aligned} \ee
where $f_i:\mathbb{R}^{d\times r} \rightarrow \mathbb{R}$ is a continuously differentiable function held privately by the $i$-th agent, and $\Mcal$ is a compact submanifold of $\R^{d\times r}$, e.g., Stiefel manifold, Oblique manifold \cite{absil2009optimization,hu2020brief,boumal2023introduction}.

While numerous algorithms \cite{bianchi2012convergence,di2016next,tatarenko2017non,wai2017decentralized,hong2017prox,zeng2018nonconvex,scutari2019distributed,sun2020improving} have been explored for decentralized optimization with nonconvex objective functions, there are only a few papers dealing with the nonconvex constraint. 
This is an important issue because there is a frequent interest in optimizing nonconvex functions over nonconvex sets, especially compact submanifolds, see, e.g., \cite{scaglione2008decentralized,ye2021deepca,chen2021decentralized,boumal2015low,mishra2019riemannian,hu2023decentralized}. 
Such nonconvex constraint introduces additional challenges in the implementation and analysis of decentralized optimization algorithms. These scenarios often require global solutions for a series of nonconvex constrained optimization problems, potentially obstructing the use of conventional tools for algorithmic complexity analysis. In this work, we overcome these issues and develop decentralized algorithms to minimize nonconvex functions over nonconvex compact submanifolds.

The Douglas-Rachford splitting (DRS) is recognized as a famous and efficient splitting algorithm in solving convex and nonconvex optimization problems. A recent study in \cite{themelis2020douglas} shows its primal equivalence with the more popular alternating direction method of multipliers (ADMM) established in \cite{themelis2020douglas}, demonstrating the convergence of both methods using the Douglas-Rachford envelope. This leads to the following question: can we design provably convergent decentralized DRS methods for solving \eqref{prob:original}?

\subsection{Our contributions}
In this paper, we leverage a novel fusion of gradient tracking, DRS, and manifold optimization, presenting two decentralized DRS algorithms to solve the decentralized manifold optimization problem \eqref{prob:original}.
\begin{itemize}
    \item \textbf{An easy-to-implement paradigm of decentralized DRS.} By utilizing the decentralized communication graph to construct an inexact projection to the consensus set, we develop a decentralized DRS method, DDRS. To mitigate potential consensus distortions caused by the nonconvexity of the manifold constraints, the communication graph needs to be well-connected. Moreover, for cases where the proximal operator of the loss function lacks a closed-form solution, we present an inexact decentralized DRS method, iDDRS, where the inexactness of evaluating the proximal operator associated with the loss function gradually decreases. Numerical results on eigenvalue problems demonstrate the superior efficacy of our algorithm compared with state-of-the-art methods. DDRS and iDDRS are the first splitting algorithms for solving decentralized manifold optimization problems.

    \item \textbf{Harnessing convex-like properties for best-known convergence complexity.} Compared to algorithms for convex constraints, the main challenge in the convergence analysis of our algorithms arises from the nonconvexity of manifold constraints. To address this, we employ a powerful property of the compact submanifold from variational analysis, called proximal smoothness. With a well-connected communication graph, we ensure that all iterations stay within a small neighborhood of the manifold. Then, by leveraging the convex-like properties of the projection operator within such neighborhood and the Douglas-Rachford envelope, we establish the global convergence of DDRS with a convergence rate of $\mathcal{O}(1/K)$. For the iDDRS, we require that the evaluation errors of the proximal operator of the loss function are summable, thus achieving the same convergence rate as its exact counterpart.
    
\end{itemize}

\subsection{Related works}

Over the past decades, decentralized optimization has attracted increasing interest due to its wide applications. For the Euclidean case (i.e., $\Mcal = \R^{d\times r}$), one seminal approach, the decentralized gradient descent (DGD) method, is explored in \cite{tsitsiklis1986distributed,nedic2009distributed,yuan2016convergence}. Its limitation in achieving convergence using fixed step sizes spurred further research. To address this issue, algorithms using local historic iterative information, such as EXTRA \cite{shi2015extra,shi2015proximal}, gradient-tracking \cite{xu2015augmented,qu2017harnessing,nedic2017achieving}, and the proximal gradient primal-dual algorithm \cite{hong2017prox} have been proposed. The connections among these algorithms are studied in \cite{chang2020distributed}. Another popular approach is the decentralized ADMM, see, e.g.,  \cite{mateos2010distributed,erseghe2011fast,shi2014linear,ling2015dlm,Maros2018On,zhang2021penalty}, which empirically converges faster than the former methods. 
These studies can only tackle the nonconvexity from $f_i$ and fail to converge if the constraint set is nonconvex. Given the equivalence between ADMM and DRS, one of our motivations is to design fast decentralized DRS algorithms for solving \eqref{prob:original} with nonconvex submanifold constraint.

For the case where $\Mcal$ is the Stiefel manifold, the authors \cite{chen2021local,chen2021decentralized} propose a decentralized Riemannian gradient descent method and its gradient-tracking version \cite{chen2021decentralized}. When the local objective function is of negative log-probability type, a decentralized Riemannian natural gradient method is proposed in \cite{hu2023decentralized}. To use a single step of consensus, augmented Lagrangian methods \cite{wang2022decentralized,wang2022variance} are proposed. For the general submanifold setting, through a theoretical study on the regularity condition of the consensus on the manifold, the authors \cite{hu2023achieving} establish the linear consensus results of the projected gradient method and Riemannian gradient method. Based on that, a decentralized projected gradient descent method and its gradient-tracking version are presented in \cite{deng2023decentralized}. As observed in the Euclidean setting, the convergence speed of these gradient-tracking type methods can be worse than the ADMM-type methods.

\subsection{Notation}
For any positive integer $n$, let $[n]={1,2,\ldots, n}$. Define $J= \frac{1}{n} \mathbf{1}_n\mathbf{1}_n^\top$, where $\mathbf{1}_n \in \R^n$ is a vector with all entries set to $1$. For a real number $a$, we use $\lceil a \rceil$ to denote the smallest integer greater than $a$. For a square matrix $W$ and an integer $t$, $W^t$ denotes the $t$-th power of $W$. Let $\bW= W \otimes I_d \in \R^{(nd) \times (nd)}$, with $\otimes$ representing the Kronecker product.  For the submanifold $\Mcal \subset \R^{d\times r}$, we always set the Euclidean metric as the Riemannian metric. We denote the tangent space and the normal space of $\Mcal$ at a point $x$ as $T_x\Mcal$ and $N_x\Mcal$, respectively. 
  
   Given $n$ agents $(x_1,\ldots,x_n)$, where $x_i\in \mathbb{R}^{d\times r}, i \in [n]$, we denote $\bx = (x_1^\top,\ldots,x_n^\top)^\top$, $\hat{\bx} =  (\hat{x}^\top,\ldots,\hat{x}^\top)^\top$ and $\bar{\bx}: =  (\bar{x}^\top,\ldots,\bar{x}^\top)^\top$, where $\hat{x}, \bar{x}$ are defined in \eqref{def:xhat} and \eqref{def:xbar}, respectively. We use $\|x\|$ to denote the Frobenius norm of $x$ and $\|\bx\|_{F,\infty} := \max_i \|x_i\|$. 
We also denote $f (\bx) =\sum_{i=1}^n f_i(x_i)$. Its Euclidean gradient is given by
$$
\begin{aligned}
\nabla f(\bx) = (\nabla f_1(x_1)^\top, \ldots, \nabla f_n (x_n)^\top)^\top,
\end{aligned}
$$
  where $\nabla f_i(x_i)$ denotes the Euclidean gradient of $f_i$ at $x_i$. If $x_i\in \Mcal, i\in [n]$, we denote the Riemannian gradient of $f$ as
  $$
  \grad f(\bx) = (\grad f_1(x_1)^\top, \ldots, \grad f_n (x_n)^\top)^\top,
  $$
  where $\grad f_i(x_i)$ denotes the Riemannian gradient of $f_i$ at $x_i$ defined in \eqref{def:rie-grad}.
  We denote the $n$-fold Cartesian product of $\Mcal$ as $\Mcal^n = \underbrace{\Mcal \times \cdots \times \Mcal}_{n}$.

\section{Preliminary}

In the context of decentralized optimization, the communication accessibility across the agents is modeled by an undirected connected network graph $G = (\mathcal{V};\mathcal{E})$ with $|\mathcal{V}| = n$. Let $W$ be the adjacency matrix of the graph. We then make the following standard assumption on $W$ \cite{zeng2018nonconvex,chen2021decentralized}.
\begin{assumption} \label{assum-w}
     We assume that the mixing matrix $W$ satisfies the following conditions:
\begin{itemize}
    \item[(i)] $W_{ij}\geq 0$ for any $i, j\in [n]$ and $W_{ij}=0$  if and only if  $(i,j)\not\in {\mathcal E}$.
    \item[(ii)] $W = W^\top$ and $W \mathbf{1}_n = \mathbf{1}_n$.
    \item[(iii)] The null space of $(I-W)$ is $\operatorname{span}(\mathbf{1}_n):=\{c {\bf 1}_n: c \in \R\}$.    
\end{itemize}
\end{assumption}
It follows from \cite{pillai2005perron} that the second largest singular value of $W$, denoted as $\sigma_2(W)$, is strictly less than 1. To simplify the notation, we use $\sigma_2$ to represent $\sigma_2(W)$. 
\subsection{Manifold optimization}
Manifold optimization has attracted much attention in the past few decades, as evident in works such as   \cite{absil2009optimization,hu2020brief,boumal2023introduction}. The goal of manifold optimization is to minimize a real-valued function over a manifold, i.e.,
\be\label{pro:manifold} \min_{x \in \Mcal} \quad h(x), \ee
where $\Mcal$ is a Riemannian manifold and $h:\Mcal \rightarrow \R$ is a real-valued function. The Riemannian gradient $\grad  h(x) \in T_x\mathcal{M}$ is the unique tangent vector satisfying
\be\label{def:rie-grad}  \left< \grad h(x), \xi \right> = dh(x)[\xi],\; \forall \xi\in T_x\mathcal{M}.\ee
If $\Mcal$ is a submanifold embedded in $\R^{d\times r}$ and the function $h$ can be extended to $\R^{d\times r}$, then the Riemannian gradient of $h$ at $x$ can be computed as 
$$\grad h(x) = \Pcal_{T_{x}\Mcal}(\nabla h(x)),$$
where $\Pcal_{T_x\Mcal}$ represents the orthogonal projection onto $T_x \Mcal$.  We say $x^*$ is a stationary point of \eqref{pro:manifold} if $\grad h(x^*) = 0$.

\subsection{Proximal smoothness} 
The notion of proximal smoothness, as introduced by \cite{clarke1995proximal}, refers to the characteristic of a closed set whereby the nearest-point projection becomes a singleton when the point is close enough to the set. Specifically, for any positive real number $\gamma$, we define the $\gamma$-tube around $\mathcal{M}$ as 
 $$
 U_{\mathcal{M}}(\gamma): = \{x:{\rm dist}(x,\mathcal{M}) <  \gamma\}.$$ 
We say a closed set $\mathcal{M}$ is $\gamma$-proximally smooth if the projection operator $\Pcal_{\mathcal{M}}(x)$ is a singleton whenever $x\in U_{\Mcal}(\gamma)$. 
It is worth noting that any compact $C^2$-submanifold of $\mathbb{R}^{d\times r}$ is a proximally smooth set \cite{clarke1995proximal,balashov2021gradient,davis2020stochastic}. For instance, the Stiefel manifold is a set that is $1$-proximally smooth. Throughout this paper, we assume that $\Mcal$ is $2\gamma$-proximally smooth. By  following the proof in \cite[Theorem 4.8]{clarke1995proximal}, a $2\gamma$-proximally smooth set $\mathcal{M}$ satisfies the following property:
\be \label{eq:lip-proj-alpha}
\left\| \Pcal_{\mathcal{M}} (x) -\Pcal_{\mathcal{M}} (y)\right\| \leq 2 \|x - y\|,~~ \forall x,y \in \bar{U}_{\mathcal{M}}(\gamma), 
\ee
where $\bar{U}_{\Mcal}(\gamma):=\{x: {\rm dist}(x,\Mcal) \leq  \gamma\}$ is the closure of $U_{\Mcal}(\gamma)$. 
Moreover, for any point $x \in \mathcal{M}$ and a normal $v \in$ $N_{x} \mathcal{M}$, it holds that
\be \label{eq:normal-bound}
\iprod{v}{y-x} \leq \frac{\|v\|}{4\gamma} \|y-x\|^2, \quad \forall y \in \mathcal{M}, 
\ee
This is often referred to as the normal inequality \cite{clarke1995proximal,davis2020stochastic}.

\subsection{The Douglas-Rachford splitting method}
The DRS is recognized as a famous and efficient splitting algorithm in solving convex and nonconvex optimization problems. Recently, its primal equivalence with the more popular ADMM is established in \cite{themelis2020douglas}.
Consider the following composite optimization problem:
\be\label{prob:two block}
\min_{x\in \mathbb{R}^p} \varphi_1(x) + \varphi_2(x),
\ee
where $\varphi_1,\varphi_2:\mathbb{R}^p \rightarrow \mathbb{R}$ are proper, lower semicontinuous, extended real-valued
functions. Starting from some $x_k,s_k,z_k\in\mathbb{R}^p$,
one iteration of the DRS applied to \eqref{prob:two block} with stepsize $\alpha$
amounts to
\begin{equation}
\left\{
    \begin{aligned}
      s_{k+1} & = s_k + z_k - x_k,\\
      x_{k+1} & = \prox_{\alpha \varphi_1}(s_{k+1}),\\
      z_{k+1} & = \prox_{\alpha \varphi_2}(2x_{k+1} - s_{k+1}),
    \end{aligned}
    \right.
\end{equation}
where $\prox_{\alpha \varphi_1}$ is a proximal operator of $\varphi_1$ defined by
\be
\prox_{\alpha \varphi_1}(x) = \arg\min_{y\in \mathbb{R}^p} \varphi_1(y) + \frac{1}{2\alpha}\| y - x\|.
\ee

The authors in \cite{li2016douglas} present the first general analysis of global convergence of the classical DRS for fully nonconvex problems where one function is Lipschitz differentiable. The authors in \cite{themelis2020douglas} consider the relaxed DRS and give a tight convergence result. Their convergence analysis is based on the Douglas-Rachford envelope (DRE), first introduced in \cite{patrinos2014douglas} for convex problems and generalized to nonconvex cases. In particular, the DRE of \eqref{prob:two block} is defined as
\begin{equation}\label{def:dre}
\begin{aligned}
  &  \varphi_{\alpha}^{\rm DR}(x):=\\
    & \min_{w\in \mathbb{R}^p} \left\{ \varphi_2(w) + \varphi_1(u) + \llangle \nabla \varphi_1(u), w-u \rrangle + \frac{1}{2\alpha}\|w - u\|^2 \right\},
    \end{aligned}
\end{equation}
where $u: = \prox_{\alpha \varphi_1}(x)$. 
They show that the DRE serves as an exact and continuously differentiable merit function for the original problem.

\subsection{Stationary point}
Let $x_1,\ldots,x_n\in \Mcal$ represent the local copies of $x$ at each agent. We denote $\hat{x}$ as their Euclidean average point, given by
\be\label{def:xhat}
\hat{x}: = \frac{1}{n} \sum_{i=1}^n x_i.
\ee
Let $\Pcal_{\Mcal}$ be the orthogonal projection to $\Mcal$. We define $\bar{x}$ is an element in $\Pcal_{\Mcal}(\hat{x})$, i.e.,
\be\label{def:xbar} \bar{x} \in  {\rm  argmin}_{y\in\Mcal}\sum_{i=1}^n \|y - x_i\|^2 = \Pcal_{\Mcal}(\hat{x}). \ee
Any element $\bar{x}$ in $\Pcal_{\Mcal}(\hat{x})$ is the induced arithmetic mean of $\{x_i\}_{i=1}^n$ on $\Mcal$ \cite{sarlette2009consensus}.
 The $\epsilon$-stationary point of problem \eqref{prob:original} is defined as follows.
\begin{definition} \label{def:station}
The set of points $\{x_1,x_2,\cdots,x_n\} \subset \Mcal$ is called an $\epsilon$-stationary point of \eqref{prob:original} if there exists an $\bar{x} \in \Pcal_{\Mcal}(\hat{x})$ such that
\[  \| \bx - \bar{\bx}\|^2 \leq \epsilon \quad {\rm and} \quad \|\grad f(\bar{\bx})\|^2 \leq \epsilon. \]
\end{definition}
In the following development, we always assure that $\hat{x} \in \bar{U}_{\Mcal}(\gamma)$. Consequently, $\Pcal_{\Mcal}(\hat{x})$ is a singleton and we have $\bar{x} = \Pcal_{\Mcal}(\hat{x})$. 

\section{A decentralized Douglas-Rachford splitting method}
 In this section, we will present a decentralized DRS method for solving \eqref{prob:original}. We first give the notations as follows. We let $x_{i,k}$ denote the $i$-agent in the $k$-iteration. Denote $\hat{x}_k = \frac{1}{n}\sum_{i=1}^n x_{i,k}$ and $\bar{x}_k$ be the projection of $\hat{x}$ onto $\Mcal$. We also denote 
$ \bx_k = (x_{1,k}^\top,\cdots,x_{n,k}^\top)^\top$, $\hat{\bx}_k = (\hat{x}_k^\top,\cdots,\hat{x}_k^\top)^\top$ and $\bar{\bx}_k = (\bar{x}_k^\top,\cdots,\bar{x}_k^\top)^\top$. By Assumption \ref{assum-w}, the equality constraint $x_1 = \cdots = x_n$ can be reformulated as $(I_{nd} - \bW) \bx = 0$, 
where $I_{nd}$ is the $nd$-by-$nd$ identity matrix.  
% Then \eqref{prob:original}  can be rewritten as
% \be \label{prob} 
% \begin{aligned}
%     \min_{\bx \in \mathbb{R}^{nd \times r}} \quad &  f (\bx) \\
%     \st \quad & (I_{nd} - \bW) \bx = 0, \; \bx \in \Mcal^n.
% \end{aligned}
% \ee
 Let us define
 $$
 \mathcal{C}:=\{ \bx \in \Mcal^n : (I_{nd} - \bW)\bx = 0 \}
 $$
 and denote $\delta_\mathcal{C}$ by the indicator function of $\mathcal{C}$. Then, problem \eqref{prob:original} can be written as 
\be \label{prob:comp} \min_{\bx \in \mathbb{R}^{nd \times r}} \quad f (\bx) + \delta_{\mathcal{C}}(\bx). \ee
Note that the nearest-point projection of a point $\bx$ to $\mathcal{C}$ has an explicit formulation, namely,
\be \label{eq:proj} \Pcal_{\mathcal{C}}(\bx) = \argmin_{\{\by \in \mathcal{M}^n: y_1 = \cdots = y_n \}} \|\by - \bx\|^2 = \Pcal_{\Mcal^n}(\hat{\bx}) =: \bar{\bx}. \ee
Given $\bs_0, \bz_0, \bx_0 \in \Mcal^n$,  at the $k$-th iterate, 
a direct application of the DRS method for solving \eqref{prob:comp} has the following update scheme:
\be\label{eq:drs-}
\left\{ \begin{aligned}
    \bs_{k+1} & = \bs_k + \bz_k - \bx_k, \\
    \bx_{k+1} & = \prox_{\alpha f}(\bs_{k+1}), \\
    \by_{k+1} & = 2 \bx_{k+1} - \bs_{k+1}, \\
    \bz_{k+1} & = \Pcal_{\mathcal{C}}(\by_{k+1}). 
\end{aligned} \right.
\ee
Let $\hat{y}_k = \frac{1}{n}\sum_{i=1}^n y_{i,k}$. The agent-wise version of \eqref{eq:drs-} can be written as: for any $i\in [n]$,
\be \label{eq:drs-agent}
\left\{ \begin{aligned}
    s_{i,k+1} & = s_{i,k} + z_{i,k} - x_{i,k}, \\
    x_{i,k+1} & = \prox_{\alpha f_i}(s_{i,k+1}), \\
    y_{i,k+1} & = 2 x_{i,k+1} - s_{i,k+1}, \\
    z_{i,k+1} & = \Pcal_{\Mcal}(\hat{y}_{k+1}). 
\end{aligned} \right.
\ee
Note that in the update of $z_{i,k+1}$, the $i$-th agent needs to collect $\{y_{i,k+1}\}_{i=1}^n$. This can be easily achieved in the centralized setting, but could be a critical issue for the decentralized setting where only partial communication along the graph is allowed.

To address the above problem in the $z$-update, a natural way is to investigate the local average instead of the global average. However, due to the existence of the nonconvexity of $\Mcal$, a careful design to control the approximation error is needed. Before introducing our approaches, let us rewrite the $z$-update in a form with a more explicit dependence on $x$. It is easily shown that the update of $x_{i,k+1}$ implies that 
$$s_{i,k+1} = x_{i,k+1} + \alpha \nabla f_i(x_{i,k+1}).$$ This together with the update of $y_{i,k+1}$ yields 
\begin{equation}
    y_{i,k+1} = x_{i,k+1} - \alpha \nabla f_i(x_{i,k+1}),\; i\in [n].
\end{equation}
Therefore, the last two rows in \eqref{eq:drs-} can be simplified as follows:
\begin{equation}
    z_{i,k+1} = \Pcal_{\Mcal}\left( \frac{1}{n} \sum_{j=1}^n \left(x_{j,k+1} - \alpha \nabla f_j(x_{j,k+1}) \right) \right). 
\end{equation}
Based on the above formulation, and utilizing the adjacency matrix $W$ under Assumption \ref{assum-w}, we approximate $\frac{1}{n}\sum_{j=1}^nx_{j,k+1}$ by $\sumj \wt x_{j,k+1}$, where $t$ is an integer, denoting the communication rounds.  To obtain a better performance, we adopt the gradient tracking techniques \cite{nedic2017achieving,qu2017harnessing,chen2021decentralized} on the gradient $-\frac{\alpha}{n}  \sum_{j=1}^n \nabla f_j(x_{j,k+1}) = \frac{1}{n}  \sum_{j=1}^n (x_{j,k+1} - s_{j,k+1}) $, i.e., 
\be \label{eq:z-track} d_{i,k+1} = \sum_{j=1}^n W_{ij}^t d_{i,k} + x_{i,k+1} - s_{i,k+1} - (x_{i,k} - s_{i,k}), \ee
where $d_{i,0}:= x_{0,k} -s_{0,k}$.
Then, these approximations give a modified and operational update:
\be\label{eq:update-z} z_{i,k+1} = \Pcal_{\Mcal}\left( \sumj \wt x_{j,k+1} + d_{i,k+1} \right). \ee

With \eqref{eq:z-track} and \eqref{eq:update-z}, our decentralized DRS method performs the following update in the $k$-th iteration, for $i=1,\ldots, n$,
\be \label{eq:drs-decen}
\left\{ \begin{aligned}
 s_{i,k+1} & = s_{i,k} + z_{i,k} - x_{i,k},\\
    x_{i,k+1} & = \prox_{\alpha f_i}(s_{i,k+1}), \\
    %y_{i,k+1} & = 2 x_{i,k+1} - s_{i,k} \\
    d_{i,k+1} & = \sum_{j=1}^n W_{ij}^t d_{j,k} + x_{i,k+1} - s_{i,k+1} - (x_{i,k} - s_{i,k}), \\
    z_{i,k+1} & = \Pcal_{\Mcal}\left(\sum_{j=1}^n W_{ij}^t x_{j,k+1} + d_{i,k+1} \right).
\end{aligned}  \right.
\ee
The detailed description is given in Algorithm \ref{alg:drgta}.
\begin{algorithm}[htbp]
\caption{Decentralized DRS method for solving \eqref{prob:original} (\textbf{DDRS})} \label{alg:drgta}
\begin{algorithmic}[1]
\REQUIRE  Initial point $\bs_0, \bz_0 \in \mathcal{M}^n$, an integer $t$, the step size $\alpha$. 
\STATE Let $x_{i,0} = s_{i,0}$, $d_{i,0} = x_{i,0} - s_{i,0}$ and $y_{i,0} = 2x_{i,0} - s_{i,0}$ on each node $i\in [n]$.
\FOR {$k=0,\cdots$ (for each node $i\in [n]$, in parallel)}
\STATE Update $s_{i,k+1}  = s_{i,k} + z_{i,k} - x_{i,k}$. 
\STATE Update $ x_{i,k+1}  = \prox_{\alpha f_i}(s_{i,k+1})$.  
\STATE Update $y_{i,k+1}  = 2 x_{i,k+1} - s_{i,k+1}$.
\STATE Perform gradient tracking:
$$
d_{i,k+1} = \sum_{j=1}^n W_{ij}^t d_{j,k} + x_{i,k+1} - s_{i,k+1} - (x_{i,k} - s_{i,k}).
$$
\STATE Update $z_{i,k+1} = \Pcal_{\Mcal}\left(\sum_{j=1}^n W_{ij}^t x_{j,k+1} + d_{i,k+1} \right)$. 
\ENDFOR
\end{algorithmic}
\end{algorithm}
As will be seen in the next section, a sufficiently large integer $t$ is important to tackle the nonconvexity from the manifold constraint. Basically speaking, a large $t$ will guarantee that the iterates remain in the proximally smooth neighborhood of $\Mcal$, which allows us to utilize the convex-like properties, \eqref{eq:lip-proj-alpha} and \eqref{eq:normal-bound}. Moreover, $\{y_{i,k}\}$ is an auxiliary sequence, which is useful for the subsequent analysis.

\section{Convergence analysis}
In this section, we present the global iteration complexity of our decentralized DRS method. %For ease of notation, let us denote  $\nabla f(\hat{\bx}): = (\nabla f_1(\hat{x}^\top,\cdots,\nabla f_n(\hat{x})^\top)^\top\in \mathbb{R}^{nd\times r}$ and $\nabla f(\bar{\bx}): = (\nabla f_1(\bar{x})^\top,\cdots,\nabla f_n(\bar{x})^\top)^\top\in \mathbb{R}^{nd\times r}$. 
 We first make the following assumptions.
\begin{assumption} \label{assum:f}
    For any given $i$, the function $f_i$ is $L$-smooth, i.e., for any $x,y\in \mathbb{R}^{d\times r}$,
    \be\label{eq:egrad-lip}
        \|\nabla f_i(x) - \nabla f_{i}(y) \| \leq L_f \|x- y\|.
    \ee
    Moreover, we assume that $\zeta: = \max_{x,y\in \Mcal} \|x - y\|$ is a finite constant. 
\end{assumption}
 Using \eqref{eq:egrad-lip}, we can readily obtain a quadratic upper bound for $f_i$: for $x,y\in {\rm conv}(\Mcal)$, it holds that
\be\label{eq:grad-Lip}
f_i(y) \leq f_i(x) + \left<\nabla f_i(x), y-x \right> + \frac{L_f}{2} \| y - x\|^2,  \;\; i\in [n].
\ee
Moreover, it follows from \cite[Lemma 4.2]{deng2023decentralized} that there exists $L>L_f$ such that
\begin{equation}\label{g-riemannian-lip}
    \| \grad f_i(x) - \grad f_i(y) \| \leq L \|x - y\|,~i\in [n],
\end{equation}

Due to the proximal smoothness of the manifold constraint, the update of the variable $z_{i,k+1}$ is well-defined only when the term $\sum_{j=1}^n W_{ij}^t x_{j,k+1} + d_{i,k+1}$ lies within the neighborhood $\bar{U}_{\Mcal}(\gamma)$, i.e., for all $k \geq 0$,
\be\label{eq:wx+d}
\sum_{j=1}^n W_{ij}^t x_{j,k+1} + d_{i,k+1} \in \bar{U}_{\Mcal}(\gamma).
\ee
Therefore, before demonstrating the main convergence result of Algorithm \ref{alg:drgta}, we will prove that \eqref{eq:wx+d} holds under some mild conditions.
Let us define several constants that will be used in the next analysis, namely, 
\be\label{def:delta}\delta_1: = \frac{\gamma}{4},\; \delta_2:= \frac{\delta_1}{12},\; \delta_3: = 2\delta_2 + \zeta,\ee
where $\gamma$ occurs in \eqref{eq:lip-proj-alpha}.  Let $\mathcal{N}_1,\mathcal{N}_2$ be two neighborhoods defined by
\be\label{def:neiborhood}
\begin{aligned}
    \mathcal{N}_1: & = \{\bx\in \mathbb{R}^{nd\times r}~: ~ \|\hat{x} - \bar{x}\| \leq \delta_1\}, \\
    \mathcal{N}_2: &= \{\bx\in \mathbb{R}^{nd\times r}~: ~ \|\hat{x} - \bar{x}\| \leq 10\delta_2\}.
\end{aligned}\ee

The next lemma demonstrates that under certain conditions on $\alpha, t$, if $\bx_0 \in \mathcal{N}_1$ and $\bs_0 \in \mathcal{N}_2$, then for all $k$, it holds that $\bx_k\in \mathcal{N}_1$ and $\sum_{j=1}^n W_{ij}^tx_{i,k} + d_{i,k}$ remains within the neighborhood $\bar{U}_{\Mcal}(\gamma)$. This latter result allows us to invoke the Lipschitz continuity \eqref{eq:lip-proj-alpha} of $\Pcal_{\Mcal}$ over $\bar{U}_{\Mcal}(\gamma)$ in the subsequent analysis. We provide the proof in Appendix \ref{appen-1}.
\begin{lemma} \label{lem:stay-x}
  Suppose that Assumption \ref{assum-w} and \ref{assum:f} hold.  Let $\{\bs_k,\bx_k,\by_k,\bz_k\}$ be generated by Algorithm \ref{alg:drgta} with 
  $$
  \begin{aligned}
  0 & < \alpha  \leq \min \left\{\frac{1}{2L}, \frac{\delta_2}{3\|\nabla f(0)\| + 2L\left(\zeta + \delta_2\right)} \right\}, \\
  t & \geq \left \lceil \max\left\{   \log_{\sigma_2}(\frac{1}{4\sqrt{n}}), \log_{\sigma_2}(\frac{\delta_3}{\delta_2 \sqrt{n}})  \right\} \right \rceil.
  \end{aligned}
  $$
  If $\|\mathbf{d}_0\|_{F,\infty} \leq 4\delta_2$,  $\|\bs_0\|_{F,\infty} \leq \zeta + \delta_2$, $\bx_0 \in \mathcal{N}_1$ and $\bz_0 \in \mathcal{N}_2$, then it holds that for any integer $k>0$,
    \be \label{bound-dk-xk-zk} \bx_k \in \mathcal{N}_1,\;  \quad {\rm and} \quad \bz_k \in \mathcal{N}_2, \ee
   where $\delta_1,\delta_2,\delta_3$ are defined in \eqref{def:delta} and $\mathcal{N}_1,\mathcal{N}_2$ are defined in \eqref{def:neiborhood}.  Moreover, we have that for any integer $k>0$,
    \begin{align}
    \sum_{j=1}^n {W_{ij}^t} x_{j,k} +d_{i,k} & \in \bar{U}_{\Mcal} (\gamma), ~ i\in [n]. \label{eq:neibohood2}
\end{align}
\end{lemma}

As shown in \cite{patrinos2014douglas}, the DRE defined in \eqref{def:dre} can serve as the potential function to analyze the convergence of the DRS method. In particular, given any $\bs$, we define $\bx = \prox_{\alpha f}(\bs)$, $\by = \bx - \alpha \nabla f (\bx)$ and $\bar{\by} = \Pcal_{\mathcal{C}}(\by)$. The DRE of \eqref{prob:comp} is defined as follows:
\be\label{eq:DR-env}
\begin{aligned}
\varphi_{\alpha}^{\rm DR}(\bs): &= f (\bx) + \min_{\bw\in \mathcal{C}} \left\{ \left<\nabla f (\bx), \bw - \bx\right> + \frac{1}{2\alpha}\|\bw - \bx\|^2
 \right\}, \\
 & = f (\bx) + \left<\nabla f (\bx), \bar{\by}  - \bx\right> + \frac{1}{2\alpha}\|\bar{\by}  - \bx\|^2.
 \end{aligned}
\ee
We then have the following descent lemma on $\varphi_{\alpha}$. We provide the proof in Appendix \ref{appen-1}.
\begin{lemma} \label{lem:descent}
     Suppose that Assumption \ref{assum:f} holds. Let $\{\bs_k,\bx_k,\by_k,\bz_k\}$ be generated by Algorithm \ref{alg:drgta}. Then, it holds that
     \be\label{eq:decent}
    \begin{aligned}
        &\varphi_{\alpha}^{\rm DR}(\bs_0) -\varphi_{\alpha}^{\rm DR}(\bs_{k+1}) 
        \geq   \sum_{\ell=0}^k \frac{1 - \alpha L - 2\alpha^2 L^2}{2\alpha} \|\bx_{\ell +1} - \bx_{\ell}\|^2 \\
        &- \frac{1 + \alpha L}{2 \alpha} \sum_{\ell=0}^k \left( \alpha \| \bx_{\ell +1} - \bx_{\ell} \|^2 + \frac{1}{\alpha} \|\bz_\ell - \bar{\by}_{\ell}\|^2 \right). 
    \end{aligned}
    \ee
\end{lemma}

Putting the above results together, we establish the following convergence rate of $\mathcal{O}(1/k)$ to reach a stationarity. We provide the proof in Appendix \ref{appen-1}.
\begin{theorem}\label{them-main}
    Suppose that Assumptions \ref{assum-w} and \ref{assum:f} hold. Let $\{\bs_k,\bx_k,\by_k,\bz_k\}$ be generated by Algorithm \ref{alg:drgta} with 
    $$
    \begin{aligned}
    0 < \alpha & \leq \min\{\frac{1}{2(1+2L+\mathcal{C}_1L^2)}, \frac{\delta_2}{3\|\nabla f(0)\| + 2L(\zeta + \delta_2)} \},\\
t & \geq  \left \lceil \max\{ \log_{\sigma_2}(\frac{1}{4\sqrt{n}}), \log_{\sigma_2}(\frac{\delta_3}{\delta_2 \sqrt{n}}), \log_{\sigma_2} \frac{1}{12\sqrt{n}}\} \right \rceil.
    \end{aligned}
    $$
    Let $f^*$ be the optimal value of \eqref{prob:comp}. If $\|\mathbf{d}_0\| \leq 4\delta_2$, $\|\bs_0\|_{F,\infty} \leq \zeta + \delta_2$, $\bx_0 \in \mathcal{N}_1$ and $\bz_0 \in \mathcal{N}_2$, for any $k\in \mathbb{N}$, it holds that 
    \be \label{eq:station-consen}
    \begin{aligned}
      &  \min_{ 0 \leq \ell \leq k} \|\bx_\ell - \bar{\bx}_{\ell}\|  \\
      \leq & 
     \frac{8\alpha(\mathcal{C}_1\alpha^2L^2+4)}{k+1}   \left(\varphi_{\alpha}^{\rm DR}(\bx_0, \bar{\by}_0) - f^* + \frac{\mathcal{C}_2}{\alpha^2} \right) + \frac{2\mathcal{C}_2}{k+1},
       \end{aligned}
         \ee
\be\label{eq:station-grad}
\begin{aligned}
       & \min_{ 0 \leq \ell \leq k} \|\grad f(\bar{\bx}_\ell) \| \\
       & \leq \frac{72(\mathcal{C}_1 \alpha^2 L^2 + 4)}{(k+1) \alpha} \left(\varphi_{\alpha}^{\rm DR}(\bx_0, \bar{\by}_0) - f^* + \frac{\mathcal{C}_2}{\alpha^2} \right)+ \frac{18 \mathcal{C}_2}{(k+1)\alpha^2},
    \end{aligned}
    \ee
    where
    \be\label{def:c-constant}
\begin{aligned}
\mathcal{C}_1: = &\frac{32}{(1-4\sigma_2^t)^2} ( 4\sigma_2^t + \frac{4}{(1-\sigma_2^t)^2}),\\
\mathcal{C}_2: = & \frac{4}{1-16\sigma_2^{2t}} \| \bz_{0} - \bar{\by}_{0} \|^2 \\
& + \frac{128}{(1-4\sigma_2^t)^2(1-\sigma_2^{2t})}\| \bd_{0} - (\hat{\bx}_{0} - \hat{\bs}_0) \|^2.
\end{aligned}
    \ee
\end{theorem}

%\onecolumn

\section{The extensions of inexact DRS}

Note that in \eqref{eq:drs-decen}, the computation of the exact proximal operator, denoted as $\prox_{\alpha f} $, is required. This computation is time-consuming in most cases. Therefore, in this section, we investigate the convergence of the algorithm when $\prox_{\alpha f} $ is computed approximately with a tolerance $\epsilon_k$. In particular, one can find $x_{i,k+1}$ satisfying
$$
x_{i,k+1} - s_{i,k+1} + \alpha \nabla f_i(x_{i,k+1}) = \mu_{i,k+1},\;\;  i\in [n],
$$
where $\|\mu_{i,k+1}\|^2 \leq  \epsilon_{k+1}, i\in [n]$.
This implies that $ x_{i,k+1}  = \prox_{\alpha f_i}(s_{i,k} +  \mu_{i,k})$. Here, we shall consider the following assumption
\begin{assumption}\label{assum:epsilon}
    $\{\epsilon_k\}_{k\in \mathbb{N}}$ is summable, i.e., $\sum_{k}\epsilon_k \leq \mathcal{D} <\infty$ for some constant $\mathcal{D}$. Moreover, $\epsilon_0 < \delta_2$, where $\delta_2$ is defined by \eqref{def:delta}.
\end{assumption}
The detailed iterative process is given in Algorithm \ref{alg:drgta2}.

\begin{algorithm}[htbp]
\caption{Inexact Decentralized DRS method for solving \eqref{prob:original} (\textbf{iDDRS})} \label{alg:drgta2}
\begin{algorithmic}[1]
\REQUIRE  Initial point $\bs_0, \bz_0 \in \mathcal{M}^n$, an integer $t$, the step size $\alpha$, the sequence $\{\epsilon_k\}$ $\epsilon_0 < \delta_2$. 
\STATE Let $x_{i,0} = s_{i,0}$, $d_{i,0} = x_{i,0} - s_{i,0}$ and $y_{i,0} = 2x_{i,0} - s_{i,0}$ on each node $i\in [n]$.
\FOR {$k=0,\cdots$ (for each node $i\in [n]$, in parallel)}
\STATE Update $s_{i,k+1}  = s_{i,k} + z_{i,k} - x_{i,k}$. 
\STATE Update  $x_{i,k+1}$:
$$x_{i,k+1}  = \prox_{\alpha f_i}(s_{i,k+1} + \mu_{i,k+1}),\; \|\mu_{i,k+1}\|^2 \leq \epsilon_k. $$ 
\STATE Update $y_{i,k+1}  = 2 x_{i,k+1} - s_{i,k+1}$.
\STATE Take gradient tracking:
$$
d_{i,k+1} = \sum_{j=1}^n W_{ij}^t d_{j,k} + x_{i,k+1} - s_{i,k+1} - (x_{i,k} - s_{i,k}).
$$
\STATE Update $z_{i,k+1} = \Pcal_{\Mcal}\left(\sum_{j=1}^n W_{ij}^t x_{j,k+1} + d_{i,k+1} \right)$. 
\ENDFOR
\end{algorithmic}
\end{algorithm}

Similar to Lemma \ref{lem:stay-x}, we have the following result. We provide the proof in Appendix \ref{appen-2}.
\begin{lemma} \label{lem:stay-x-inex}
    Suppose that Assumption \ref{assum-w}-\ref{assum:epsilon} hold.    Let $\{\bs_k,\bx_k,\by_k,\bz_k\}$ be generated by Algorithm \ref{alg:drgta2} with
    $$
  \begin{aligned}
  0 &< \alpha \leq \min\{\frac{1}{2L}, \frac{\delta_2 - \epsilon_0}{3\|\nabla f(0)\| + 2L\left(\zeta + \delta_2\right) + 2\epsilon_0} \}, \\
t & \geq \left \lceil \max\left\{  \log_{\sigma_2}(\frac{1}{4\sqrt{n}}) , \log_{\sigma_2}(\frac{\delta_3}{\delta_2 \sqrt{n}})  \right\} \right \rceil.
  \end{aligned}
  $$
  If  $\bx_0 \in \mathcal{N}_1$ and $\bz_0 \in \mathcal{N}_2$, then it holds that for any integer $k>0$,
    \be \label{bound-dk-xk-zk-inex} \bx_k \in \mathcal{N}_1,\;  \quad {\rm and} \quad \bz_k \in \mathcal{N}_2, \ee
    where $\delta_1,\delta_2,\delta_3$ are defined in \eqref{def:delta} and $\mathcal{N}_1,\mathcal{N}_2$ are defined in \eqref{def:neiborhood}. Moreover, we have that for any integer $k>0$,
    \begin{align}
    \sum_{j=1}^n {W_{ij}^t} x_{j,k} +d_{i,k} & \in \bar{U}_{\Mcal} (\gamma), ~ i\in [n]. \label{eq:neibohood2-index}
\end{align}
\end{lemma}

Denote $\mathbf{\mu}_k = (\mu_{1,k}^\top,\cdots,\mu_{n,k}^\top)^\top$. The following lemma is similar to Lemma \ref{lem:descent}, we omit the proof.
\begin{lemma} \label{lem:descent-inex}
    Suppose that Assumption \ref{assum:f} holds. Let $\{\bs_k,\bx_k,\by_k,\bz_k\}$ be generated by Algorithm \ref{alg:drgta2}. Then, 
    \be \label{eq:descent-index}  \begin{aligned}
        &\varphi_{\alpha}^{\rm DR}(\bs_0) -\varphi_{\alpha}^{\rm DR}(\bs_{k+1}) 
        \geq   \sum_{\ell=0}^k \frac{1 - \alpha L - 2\alpha^2 L^2}{2\alpha} \|\bx_{\ell +1} - \bx_{\ell}\|^2 \\
        &- \frac{1 + \alpha L}{2 \alpha} \sum_{\ell=0}^k \left( \alpha \| \bx_{\ell +1} - \bx_{\ell} \|^2 + \frac{2}{\alpha} (\|\bz_l - \bar{\by}_l\|^2 + 2\sqrt{n}\epsilon_{\ell}) \right). 
    \end{aligned} \ee
\end{lemma}

The following theorem show that Algorithm \ref{alg:drgta2} achieves the convergence rate of $\mathcal{O}(1/k)$. We provide the proof in Appendix \ref{appen-2}.

\begin{theorem}\label{theo:iter:inexact}
 Suppose that Assumption \ref{assum-w}-\ref{assum:epsilon} hold. Let $\{\bs_k,\bx_k,\by_k,\bz_k\}$ be generated by Algorithm \ref{alg:drgta2} with 
    $$
    \begin{aligned}
    \alpha & \leq \min\{\frac{1}{2(1+2L+2\mathcal{C}_3L^2)}, \frac{\delta_2 - \epsilon_0}{3\|\nabla f(0)\| + 2L\left(\zeta + \delta_2\right) + 2\epsilon_0} \},\\
    t & \geq \left \lceil \max\{ 2\log_{\sigma_2}(\frac{1}{n}), \log_{\sigma_2}(\frac{\delta_3}{\delta_2 \sqrt{n}}), \log_{\sigma_2} \frac{1}{12\sqrt{n}}\} \right \rceil.
    \end{aligned}
    $$  Let $f^*$ be the optimal value of \eqref{prob:comp}. If $\|\mathbf{d}_0\| \leq 4\delta_2$, $\|\bs_0\|_{F,\infty} \leq \zeta + \delta_2$, $\bx_0 \in \mathcal{N}_1$ and $\bz_0 \in \mathcal{N}_2$, it holds that for any $k\in \mathbb{N}$
       \be \label{eq:station-consen-index}
    \begin{aligned}
      &  \min_{ 0 \leq \ell \leq k} \|\bx_\ell - \bar{\bx}_{\ell}\|  \\
      \leq & 
     \frac{8\alpha(\mathcal{C}_1\alpha^2L^2+4)}{k+1}   \left(\varphi_{\alpha}^{\rm DR}(\bx_0, \bar{\by}_0) - f^* + \frac{\mathcal{C}_3}{\alpha^2} \right) + \frac{2\mathcal{C}_5}{k+1},
       \end{aligned}
         \ee
\be\label{eq:station-grad-index}
\begin{aligned}
       & \min_{ 0 \leq \ell \leq k} \|\grad f(\bar{\bx}_\ell) \| \\
       & \leq \frac{72(\mathcal{C}_3 \alpha^2 L^2 + 4)}{(k+1) \alpha} \left(\varphi_{\alpha}^{\rm DR}(\bx_0, \bar{\by}_0) - f^* + \frac{\mathcal{C}_5}{\alpha^2} \right)+ \frac{18 \mathcal{C}_5}{(k+1)\alpha^2},
    \end{aligned}
    \ee
        where 
   \be\label{def:constant-index}
\begin{aligned}
       \mathcal{C}_3 & :=  \frac{128}{(1-4\sigma_2^t)^2} (\sigma_2^{2t} \alpha^2 L^2 +  \frac{\alpha^2 L^2}{(1-\sigma_2^{t})^2})\mathcal{D},  \\
     \mathcal{C}_4 & :=  \frac{512n + 128n(1-\sigma_2^{t})^2}{(1-4\sigma_2^t)^2(1-\sigma_2^{t})^2} \sum_{\ell = 0}^k \epsilon_{\ell} \\
     & + \frac{128}{(1-4\sigma_2^t)^2(1-\sigma_2^{2t})}\| \bd_{0} - (\hat{\bx}_{0} - \hat{\bs}_0) \|^2 \\
    &+ \frac{4}{1-16\sigma_2^{2t}} \| \bz_{0} - \bar{\by}_{0} \|^2,\\
    \mathcal{C}_5& : = \mathcal{C}_4 + 2\sqrt{n}\mathcal{D}.
\end{aligned}
   \ee
   
\end{theorem}

\section{Numerical experiments}
In this section, we present numerical comparisons of our proposed methods with the existing decentralized manifold optimization algorithms, DRGTA \cite{chen2021decentralized} and DPGTA \cite{hu2023decentralized}, on the decentralized principal component analysis (DPCA).

The DPCA problem can be mathematically formulated as
\be \label{prob:dpca} \min_{\bx \in \Mcal^n} \;\;  -\frac{1}{2} \sum_{i=1}^n \tr(x_i^\top A_i^\top A_i x_i), \;\; \st \;\; x_1 = \cdots = x_n, \ee
where $\Mcal^n:=\underbrace{{\rm St}(d,r) \times \cdots \times {\rm St}(d,r)}_{n}$, $A_{i} \in \mathbb{R}^{m_{i} \times d}$ is the local data matrix in $i$-th agent with $m_{i}$ samples.
Note that for any solution $x^*$ of \eqref{prob:dpca}, $x^*Q$ with an orthogonal matrix $Q \in \R^{r\times r}$ is also a solution. We use the function
\[ d_s(x, x^*) := \min_{Q\in \R^{r\times r},\; Q^\top Q = QQ^\top = I_r} \; \|xQ - x^*\| \]
to compute the distance between two points $x$ and $x^*$. As $x$ is constrained on the Stiefel manifold, it always holds that $\tr(x^\top x) = r$. Then, we can define $f_i = -\frac{1}{2} \tr(x^\top (A_i^\top A_i -\|A_i\|_2^2 I)x)$ for \eqref{prob:dpca}. Consequently, the proximal operator $\prox_{\alpha f_i}$ can be exactly calculated via solving linear equations, i.e., $\prox_{\alpha f_i}(x) = (I + t(\|A_i\|_2^2 I - A_i^\top A_i))^{-1} x$. Moreover, we can save $(I + t(\|A_i\|_2^2 I - A_i^\top A_i))^{-1}$ to significantly reduce the computational costs. Given these, we only present the numerical results of DDRS.

\begin{figure}[!htb]
    \centering
    \subfloat{
        \includegraphics[width=0.23\textwidth]{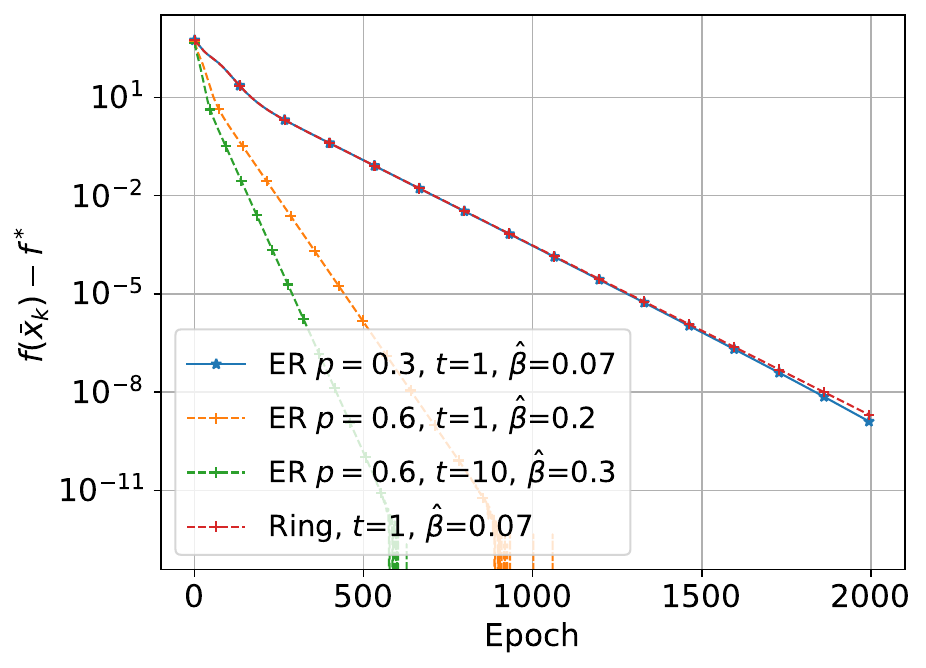}
    }
    \subfloat{
        \includegraphics[width=0.23\textwidth]{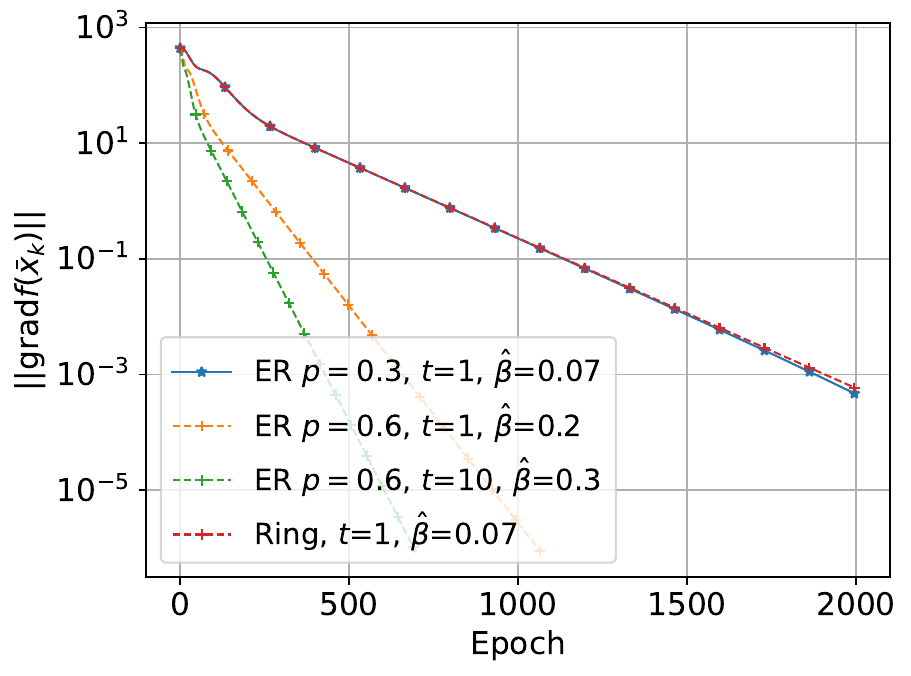}
    }\\
    \subfloat{
        \includegraphics[width=0.23\textwidth]{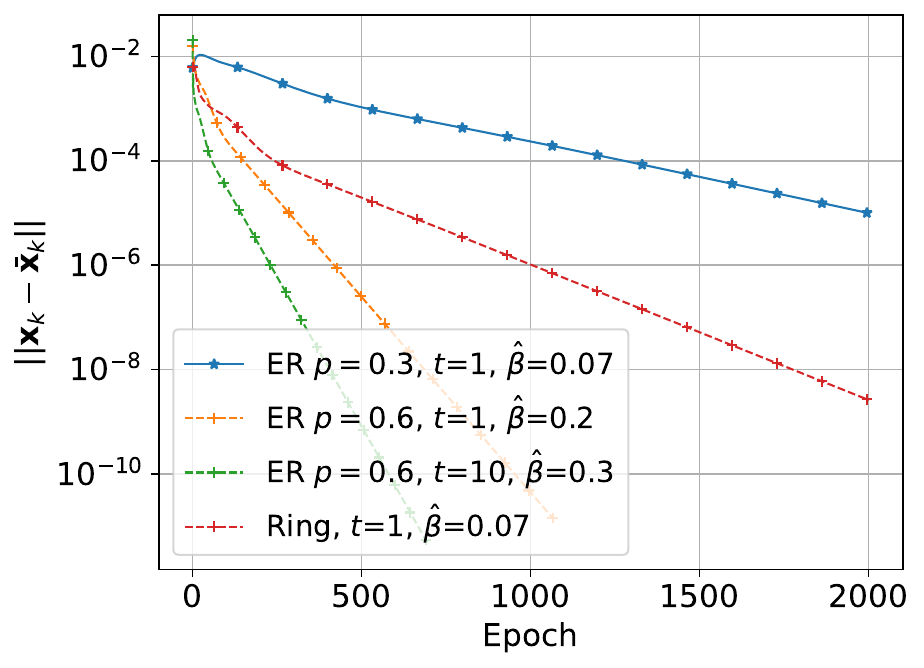}
    }
    \subfloat{
        \includegraphics[width=0.23\textwidth]{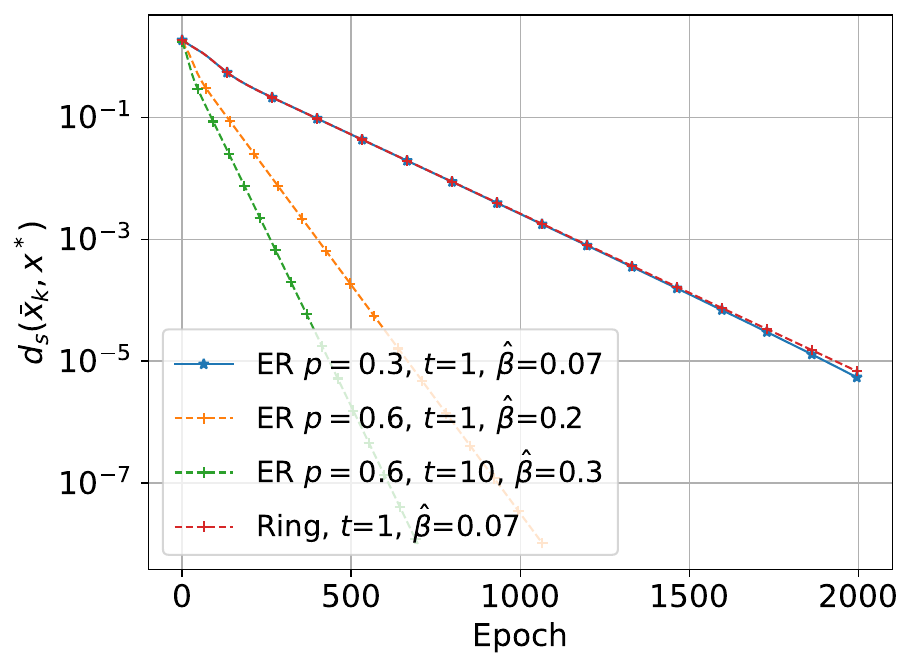}
    }
    \caption{Numerical results of DDRS for solving DPCA on the synthetic dataset with different network graphs and different $t$.}
    \label{fig:diff-DPCA}
\end{figure}

\begin{figure}[!htb]
    \centering
    \subfloat{
        \includegraphics[width=0.23\textwidth]{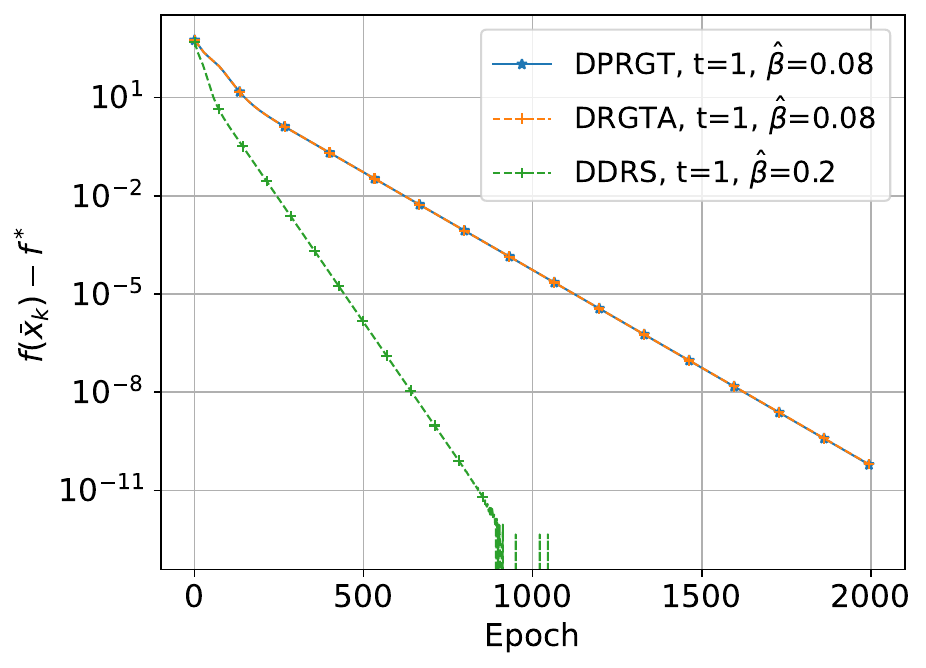}
    }
    \subfloat{
        \includegraphics[width=0.23\textwidth]{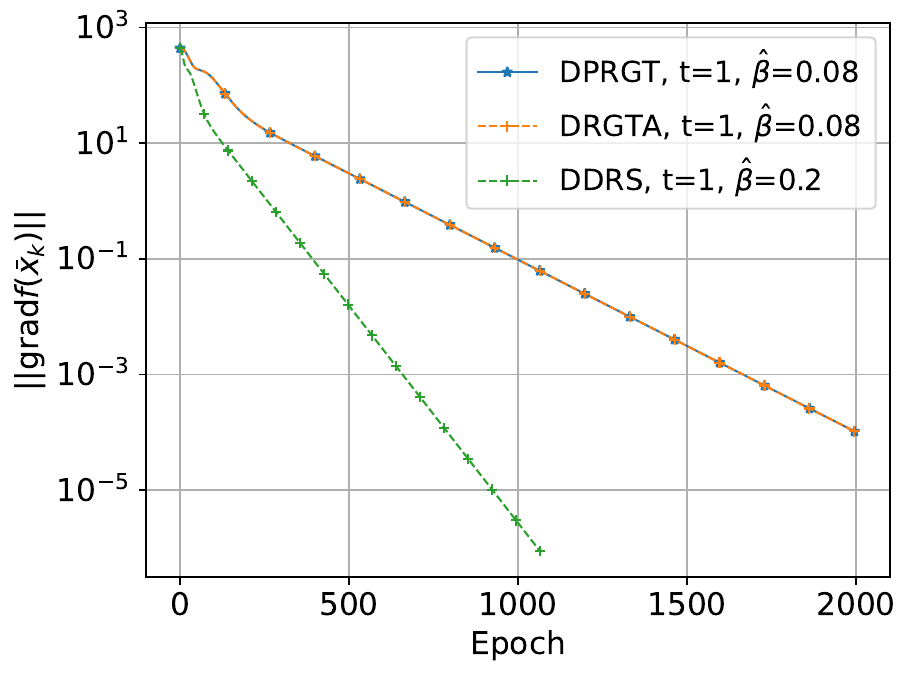}
    }\\
    \subfloat{
        \includegraphics[width=0.23\textwidth]{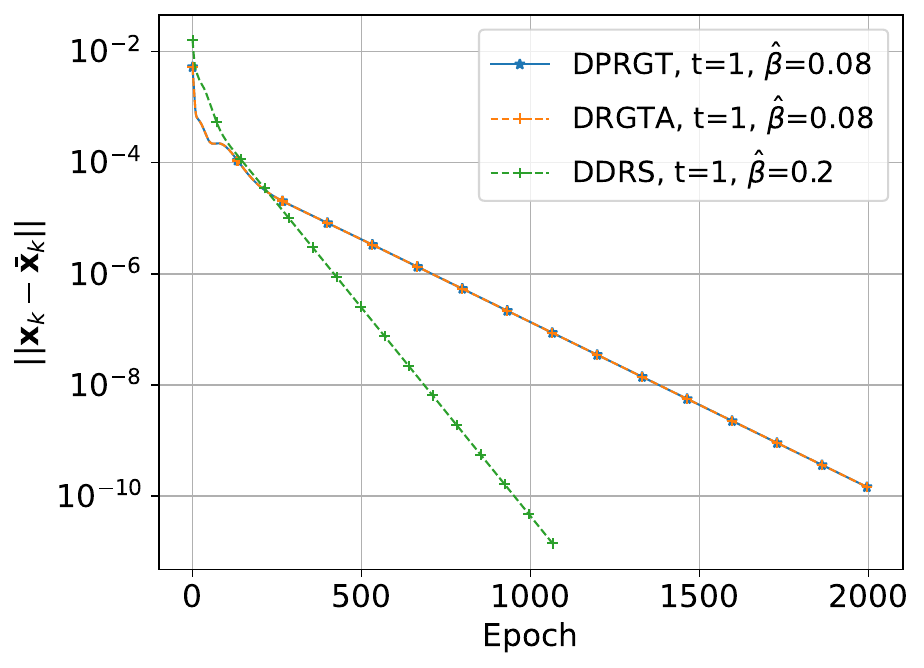}
    }
    \subfloat{
        \includegraphics[width=0.23\textwidth]{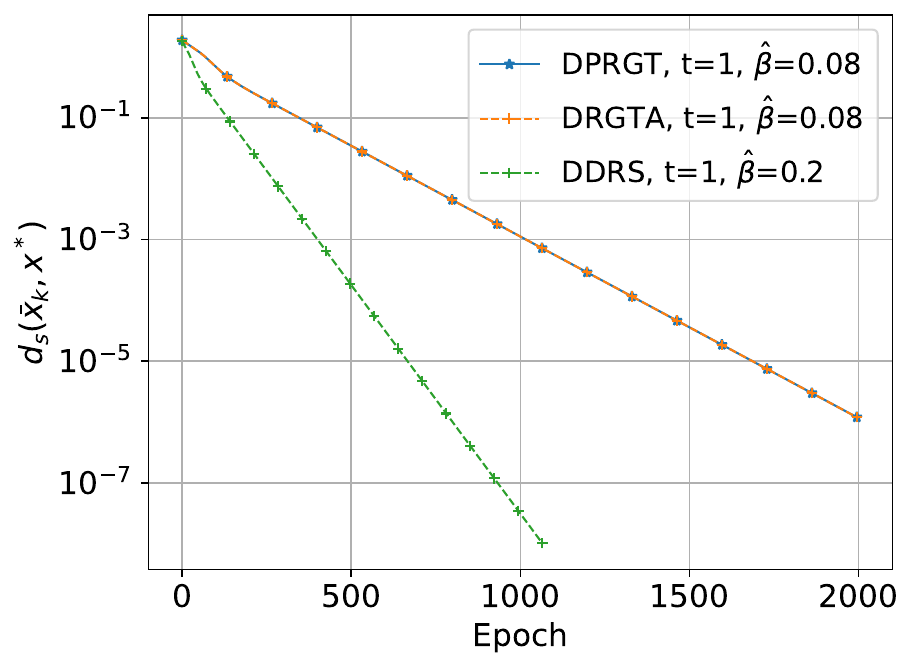}
    }
    \caption{Numerical results of different algorithms for solving DPCA on the synthetic dataset.}
    \label{fig:diff-algo}
\end{figure}

\subsection{Synthetic dataset}
We set $m_{1}=\cdots=m_{n}=1000, d=10$, and $r=5$. Then, a matrix $B \in \R^{1000n  \times d}$ is generated from the singular value decomposition
\[ B = U \Sigma V^\top, \]
where $U \in \R^{1000n \times d}$ and $V \in \R^{d\times d}$ are orthogonal matrices, and $\Sigma \in \R^{d \times d}$ is a diagonal matrix. To control the distributions of the singular values, we set $\tilde{\Sigma} = {\rm diag} (\xi^j)$ with $\xi \in (0,1)$. Furthermore, $A$ is set as
\[ A = U \tilde{\Sigma} V^\top \in \R^{1000n \times d}. \]
$A_i$ is obtained by randomly splitting the rows of $A$ into $n$ subsets with equal cardinalities. It is easy to check the first $r$ columns of $V$ form the solution of \eqref{prob:dpca}. In the experiments, we set $\xi$ and $n$ to $0.8$ and $8$, respectively.

Each algorithm employs a fixed step size $\alpha=\frac{\hat{\beta}n}{\sum_{i=1}^n m_i}$, and the grid search is carried out to determine the optimal $\hat{\beta}$. Both DRPGT and DRGTA utilize the polar decomposition for their retraction operations. The connectivity between agents is simulated using various graph matrices like the Erdos-Renyi (ER) network with probability settings $p = 0.3, 0.6$ and the Ring network. In addition, we opted for the Metropolis constant edge weight matrix \cite{shi2015extra} as our choice of the mixing matrix $W$.

The results are presented in Figures \ref{fig:diff-DPCA} and \ref{fig:diff-algo}. It can be seen from Figure \ref{fig:diff-DPCA} that
DDRS with the multiple-step consensus (i.e., $t=10$) allows using a larger step size and converges faster than those with the single-step consensus (i.e., $t=1$). Besides, a denser graph (e.g., ER $p=0.6$) will give better solutions. For Figure \ref{fig:diff-algo}, we see that DDRS converges much faster than DPRGT and DRGTA, and
DPRGT and DRGTA have very close trajectories on the consensus error, the objective function, the gradient norm, and the distance to the global optimum.

\begin{figure}[!htb]
    \centering
    \subfloat{
        \includegraphics[width=0.23\textwidth]{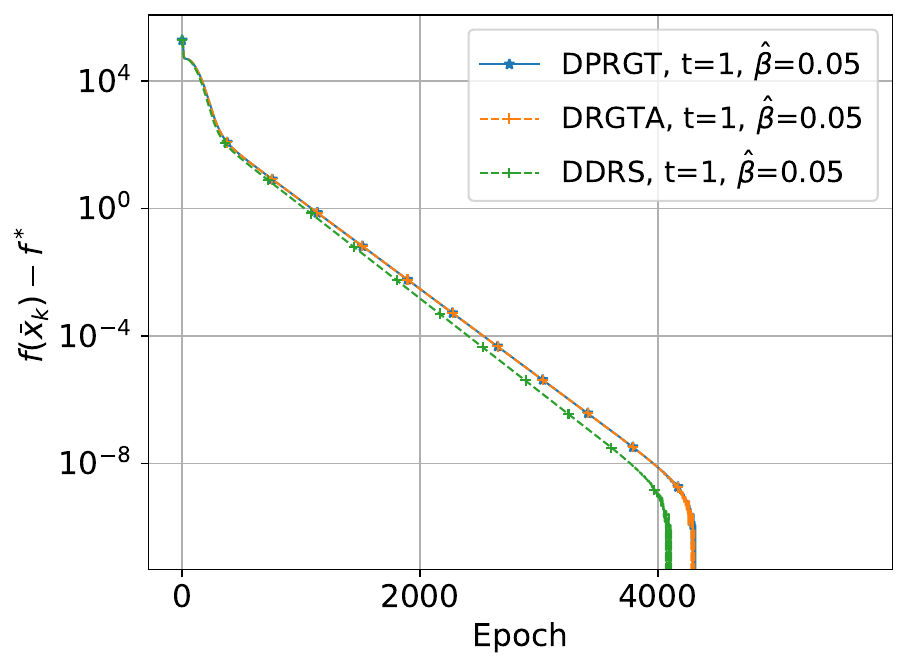}
    }
    \subfloat{
        \includegraphics[width=0.23\textwidth]{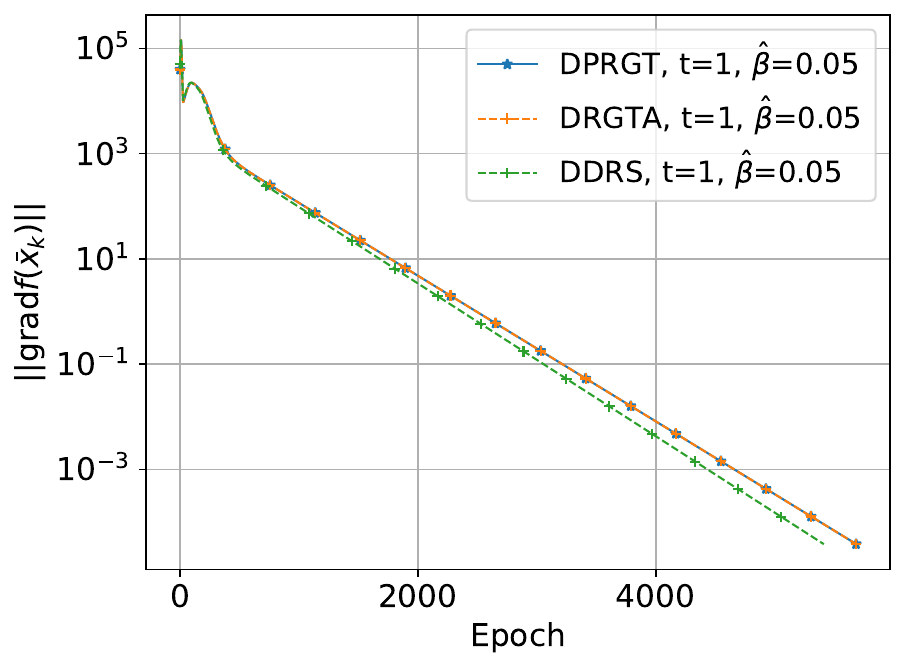}
    }\\
    \subfloat{
        \includegraphics[width=0.23\textwidth]{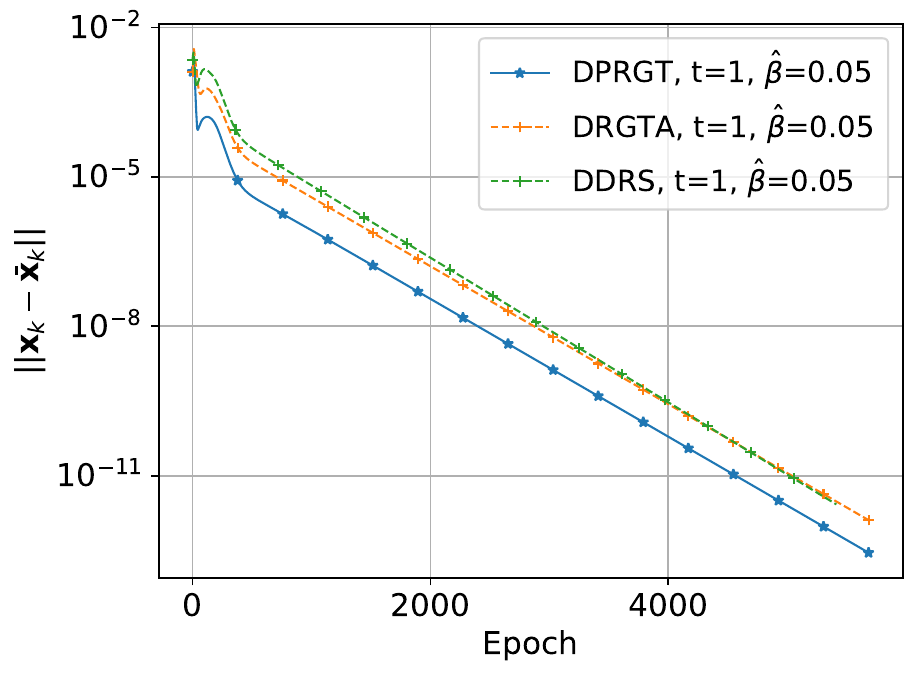}
    }
    \subfloat{
        \includegraphics[width=0.23\textwidth]{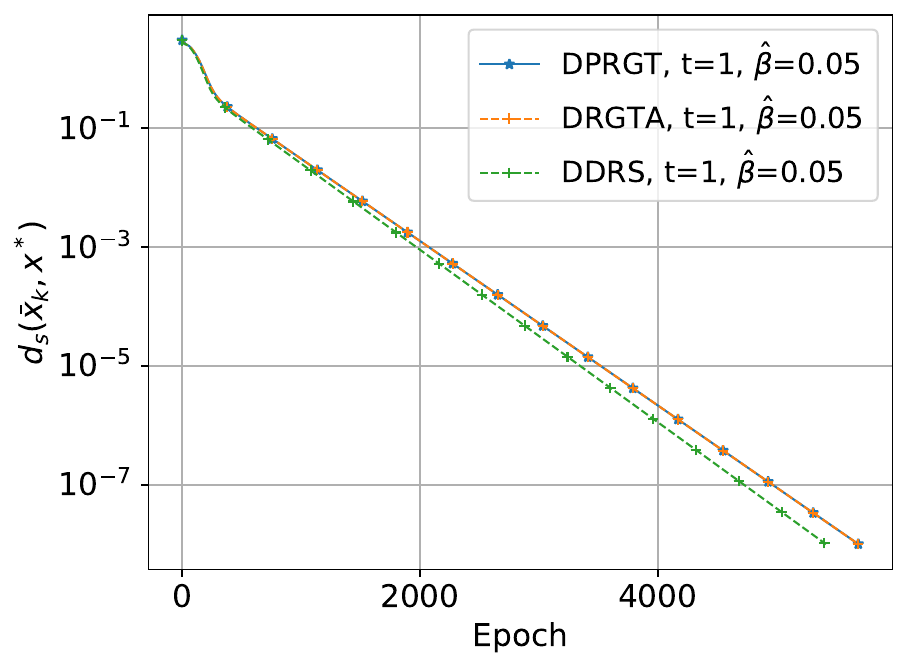}
    }\\
    \caption{Numerical results of different algorithms for solving DPCA on the Mnist dataset.}
    \label{fig:diff-algo-mnist}
\end{figure}

\subsection{Mnist dataset}
To evaluate the efficiency of our proposed method, we also conduct numerical tests on the Mnist dataset \cite{lecun1998mnist}. The testing set, consisting of 60000 handwritten images of size $32 \times 32$, is used to generate $A_i$'s. We first normalize the data matrix by dividing 255 and randomly split the data into $n=8$ agents with equal cardinality. Then, each agent holds a local matrix $A_i$ of dimension $\frac{60000}{n} \times 784$. We compute the first 5 principal components, i.e., $d=784, r=5$.

For all algorithms, we use the fixed step sizes $\alpha = \frac{\hat{\beta}}{60000}$ with a best-chosen $\hat{\beta}$.
Echoing the findings from the synthetic dataset, Figure \ref{fig:diff-algo-mnist} reveals that DDRS consistently outperforms both DPRGT and DRGTA, with the latter two showcasing very similar behaviors.

\section{Conclusion}
Through a novel fusion of gradient tracking, DRS, and manifold optimization, we propose two efficient decentralized Douglas-Rachford splitting algorithms, DDRS and iDDRS, for solving decentralized smooth optimization problems on compact submanifolds. To address the nonconvexity challenge of the manifold constraint, we employ the fundamental concept of proximal smoothness of the compact submanifold and establish the best-known convergence rate $\mathcal{O}(1/K)$ for both algorithms. Numerical experiments validate the superior performance of DDRS. Our work can also be readily extended to solve more general decentralized nonsmooth optimization problems, e.g., the strong prox-regular function \cite{hu2023projected}. 

\appendices

\section{Useful lemmas}

For a $L$-smooth $h$, its proximal operator ${\rm prox}_{\alpha h}$ is strongly monotone and cocoercive for small enough $\alpha$ \cite[Proposition 2.3]{themelis2020douglas}.
\begin{proposition} \label{prop}
    Let $h$ be a $L$-smooth function and $0<\alpha < 1/L$. Then, ${\rm prox}_{\alpha h}$ is $(\frac{1}{1+ \alpha L})$-strongly monotone and $(1- \alpha L)$-cocoercive, in the sense that
    \[ \iprod{x- x^\prime}{s - s^\prime} \geq \frac{1}{1 + \alpha L} \|s - s^\prime\|^2 \] 
    and 
    \[ \iprod{x - x^\prime}{s - s^\prime} \geq (1- \alpha L)\|x - x^\prime\|^2  \]
    for all $s, s^\prime$, where $x = {\rm prox}_{\alpha h}(s)$ and $x^\prime = {\rm prox}_{\alpha h}(s^\prime)$. In particular,
    \be \label{eq:prox-iso} 
    \frac{1}{1 + \alpha L}\|s - s^\prime\| \leq \|x - x^\prime\| \leq \frac{1}{1 - \alpha L}\|s - s^\prime\|.
    \ee
\end{proposition}

We give the following bound on the distance between ${\rm prox}_{\alpha f}(x)$ and $x$ for a $L$-smooth $h$ and $\alpha < 1/(2L)$.
\begin{lemma} \label{lem:bound-xs-inex}
Let $f$ be a $L$-smooth function and $0 < \alpha <1/(2L)$. Then, for any $x = {\rm prox}_{\alpha f}(s + \mu)$ with $\|\mu\| \leq \epsilon$,
it holds that
\be \label{eq:bound-xs-2-inex} \|s - x\| \leq \alpha \left (3\|\nabla f(0)\| +   2L\|s\| + 2\epsilon  \right) + \epsilon. \ee
When $\epsilon = 0$, it reduced to
\be \label{eq:bound-xs-2} \|s - x\| \leq \alpha \left (3\|\nabla f(0)\| +   2L\|s\| \right). \ee
\end{lemma}
\begin{proof}
   We only prove \eqref{eq:bound-xs-2-inex}. It follows from the definition of ${\rm prox}_{\alpha f}$ that
    \[ x - s + \alpha \nabla f(x) = \mu. \]
    Then, by using the $L$-smoothness of $f$ and the triangle inequality, we have
    \[ \begin{aligned}
        \|x\| & = \|s - \alpha \nabla f(x)  + \mu\| \leq \|s\| + \epsilon + \alpha (\|\nabla f(0)\| + L\|x\|).
    \end{aligned} \]
    This gives $\|x\| \leq (\|s\| +\epsilon+ \alpha \|\nabla f(0)\| )/(1 - \alpha L)$. Therefore, 
    \[ \begin{aligned}
        \|s - x\| & = \alpha \|\nabla f(x)\| + \epsilon \leq \alpha (\| \nabla f(0) \| + L \|x\|) + \epsilon \\
        & \leq \alpha \left (\|\nabla f(0)\| +  \frac{L\|s\| +\epsilon+ \| \nabla f(0)\| }{1 - \alpha L} \right) + \epsilon \\
        & \leq \alpha \left (3\|\nabla f(0)\| +   2L\|s\| + 2\epsilon  \right) + \epsilon.
    \end{aligned}
      \]
      The proof is completed. 
\end{proof}

\section{Convergence of DDRS}\label{appen-1}

\subsection{Proof of Lemma \ref{lem:stay-x}}

Building upon Lemma \ref{lem:bound-xs-inex}, we can establish that in Algorithm \ref{alg:drgta}, when $\|\bs_0\|_{F,\infty}$ is bounded by a specific constant, it follows that both $\|\bs_k\|_{F,\infty}$ and $\|\bx_k -\bar{\bx}_k \|_{F,\infty}$ are also bounded, for all $k>0$. Furthermore, one can shows that $\|\nabla f(x_{i,k})\|$ and $\|d_{i,k}\|$ are bounded.

\begin{lemma}\label{lem:bound-sk-xksk}
Suppose that Assumption \ref{assum-w} and \ref{assum:f} hold.  Let $\{\bs_k,\bx_k,\by_k,\bz_k\}$ be generated by Algorithm \ref{alg:drgta}. Let $\delta_2$ be defined in \eqref{def:delta}. If  $0 < \alpha \leq \min\{\frac{1}{2L}, \frac{\delta_2}{3\|\nabla f(0)\| + 2L\left(\zeta + \delta_2\right)} \}$, $t \geq  \left\lceil \log_{\sigma_2}(\frac{1}{4\sqrt{n}}) \right \rceil$, $\|\mathbf{d}_0\|_{F,\infty} \leq 4\delta_2$ and $\|\bs_0\|_{F,\infty} \leq \zeta + \delta_2$, then it holds that for any $k$,
    \begin{align}
    \|\bs_{k}\|_{F,\infty} \leq \zeta + \delta_2, \; \|\bx_k & - \bs_k\|_{F,\infty} \leq \delta_2, \label{eq:bound-sk} \\
        \|\nabla f(x_{i,k})\| &\leq  \delta_2 / \alpha, \; i\in [n], \label{eq:bound-nabla-f} \\
         \|d_{i,k}\| &\leq 4\delta_2, \;\;\; i\in [n].\label{eq:bound-dk}
    \end{align}
\end{lemma}

      \begin{proof}
       Firstly, we prove \eqref{eq:bound-sk} by induction. Note that $\|\bs_0\|_{F,\infty} \leq \zeta + \delta_2$ and by Lemma \ref{lem:bound-xs-inex},
    \be \label{bound-xi0-si0}
     \begin{aligned}
        \|x_{i,0} - s_{i,0}\| & \leq \alpha \left( \|\nabla f(0)\| + \frac{L\|s_{i,0}\| + \|\nabla f(0)\| }{1 - \alpha L} \right) \\ 
        & \leq \alpha \left(3\|\nabla f(0)\| + 2L\left(\zeta + \delta_2\right) \right) \leq \delta_2,
    \end{aligned}
    \ee
    which implies that $\|\bx_0 - \bs_0\|_{F,\infty} \leq \delta_2$. Suppose for some $k\geq 0$ that $\|\bs_k\|_{F,\infty} \leq \zeta + \delta_2$ and $\|\bx_k - \bs_k\|_{F,\infty} \leq \delta_2$.  Then we have 
     \[ \|\bs_{k+1}\|_{F,\infty} \leq \|\bz_k \|_{F,\infty} + \|\bx_k - \bs_k\|_{F,\infty} \leq \zeta + \delta_2. \]
     Similar with \eqref{bound-xi0-si0},  we further have that $ \|x_{i,k+1} - s_{i,k+1}\| \leq \delta_2$, for any $i \in [n]$, which implies $\|\bx_k  - \bs_k\|_{F,\infty} \leq \delta_2$.  Secondly, \eqref{eq:bound-nabla-f} follows from the fact that $x_{i,k} = s_{i,k} - \alpha \nabla f(x_{i,k})$. 
 Finally,  we prove \eqref{eq:bound-dk} by induction.  Suppose for some $k\geq 0$ such that \eqref{eq:bound-dk} holds.  It follows from \eqref{eq:z-track} and $d_{i,0} = x_{i,0} - s_{i,0}$ that
       \be\label{eq:hatd=hatxs}
       \hat{d}_k  = \hat{x}_k - \hat{s}_k = \alpha \hat{g}_k,
       \ee
       where $\hat{g}_k: = \frac{1}{n}\sum_{i=1}^n \nabla f(x_{i,k})$. Then we have that
    \begin{equation}
\begin{aligned}
&\| d_{i,k+1} - \alpha \hat{g}_k\| \\
& = \| \sum_{j=1}^n W_{ij}^td_{j,k} - \alpha \hat{g}_k +\alpha \nabla f_i(x_{i,k+1}) - \alpha\nabla f_i(x_{i,k}) \| \\
&  \overset{\eqref{eq:hatd=hatxs}}{\leq}  \| \sum_{j=1}^n (W_{ij}^t - \frac{1}{n})  d_{j,k} \|  + \alpha\|\nabla f_i(x_{i,k+1}) - \nabla f_i(x_{i,k}) \| \\
& \leq \sigma_2^t \sqrt{n} \max_i \| d_{i,k}\| +2 \delta_2 \\
% & \leq \sigma_2^t \sqrt{n} \max_i\| d_{i,k}\| + 2\delta_2 \\
% & \leq \sigma_2^t \sqrt{n} \max_i\| d_{i,k}\| + 2\delta_2 \\
& \leq \frac{1}{4}\max_i \| d_{i,k}\| + 2\delta_2  \leq 3\delta_2,
\end{aligned}
\end{equation}
where the second inequality follows from \eqref{eq:bound-nabla-f} and the bound on the total variation distance between any row of $W^t$ and $\frac{1}{n}\mathbf{1}$ \cite{diaconis1991geometric,boyd2004fastest}, i.e., 
\be\label{eq:bound:W-1}\max_i \sum_{j=1}^n |W_{i,j}^t - \frac{1}{n}| \leq \sqrt{n} \sigma_2^t. \ee
Hence,
\be
\begin{aligned}
\| d_{i,k+1}   \| & \leq\| d_{i,k+1} - \alpha \hat{g}_k\| + \|\alpha \hat{g}_k\| \\
&\leq  3\delta_2 + \alpha \max_{i}\|\nabla f(x_{i,k})\| \overset{\eqref{eq:bound-nabla-f}}{\leq}3\delta_2 + \delta_2 \leq 4\delta_2.
\end{aligned}
\ee  The proof is completed.
\end{proof}

The following lemma shows that when $\bx_k \in \mathcal{N}_1$, the term $ \sum_{j=1}^n {W_{ij}^t} x_{j,k} +d_{i,k}$ will be lie in the neighborhood $\bar{U}_{\Mcal} (\gamma)$.

\begin{lemma} \label{lem:stay-neighborhood}
 Suppose that Assumption \ref{assum-w} and \ref{assum:f} hold. Let $\{\bs_k,\bx_k,\by_k,\bz_k\}$ be generated by Algorithm \ref{alg:drgta}.    Under the same condition on $\alpha, \mathbf{d}_0$ and $\bs_0$ as in Lemma \ref{lem:bound-sk-xksk}, if $t\geq \left \lceil \max\{ \log_{\sigma_2}(\frac{1}{4\sqrt{n}}), \log_{\sigma_2}(\frac{\delta_3}{\delta_2 \sqrt{n}}) \} \right \rceil$ and $\bx_k \in \mathcal{N}_1$, then it holds that
\begin{align}
    \sum_{j=1}^n {W_{ij}^t} x_{j,k} +d_{i,k} & \in \bar{U}_{\Mcal} (\gamma), ~ i\in [n], \label{eq:neibohood1}
\end{align}
where $\delta_1,\delta_2,\delta_3$ are defined in \eqref{def:delta} and $\mathcal{N}_1$ is defined in \eqref{def:neiborhood}.
\end{lemma}

\begin{proof}
    It follows from the updated rule of $x_{i,k+1}$ and $s_{i,k+1}$ in Algorithm \ref{alg:drgta} that
     \be\label{eq:relation:xk+1-zk} \begin{aligned}
         x_{i,k+1} & = x_{i,k+1} - s_{i,k+1} + s_{i,k+1} \\ & = x_{i,k+1} - s_{i,k+1} + s_{i,k} - x_{i,k} + z_{i,k} \\
         & = z_{i,k} + \alpha \nabla f(x_{i,k})  - \alpha \nabla f(x_{i,k+1}).
     \end{aligned}
       \ee
       Then we have that 
\be\label{eq:w-xjk-hatx}
\begin{aligned}
\|x_{i,k+1}\| & \leq \|x_{i,k+1} - z_{i,k}\| + \|z_{i,k}\| \\
& \leq \alpha \| \nabla f_i(x_{i,k+1}) - \nabla f_i(x_{i,k}) \| + \|z_{i,k}\| \\
& \leq  2\delta_2 + \zeta =\delta_3,
\end{aligned}
\ee
where the last inequality follows from \eqref{eq:bound-nabla-f} and Assumption \ref{assum:f}.
 It follows that for any $i\in [n]$,
\[ \begin{aligned}
&\| \sum_{j=1}^n W_{ij}^t x_{j,k} +  d_{i,k} - \bar{x}_k  \| \\
 &\leq    \| \sum_{j=1}^n W_{ij}^t x_{j,k} + d_{i,k} - \hat{x}_k  \| + \| \hat{x}_k - \bar{x}_k \| \\
 & \leq \|\sum_{j=1}^n (W_{ij}^t - \frac{1}{n}) x_{j,k} \| + \|d_{i,k}\| +\| \hat{x}_k - \bar{x}_k \| \\ 
 & \leq   \sum_{j=1}^n \left|W_{ij}^t - \frac{1}{n}\right| \max_j \| x_{j,k}\| +4 \delta_2 + \delta_1 \\ 
 & \leq \sqrt{n} \sigma_2^t  \delta_3 + 4 \delta_2 + \delta_1\\
 & \leq 5\delta_2 + \delta_1 \leq \frac{17}{12}\delta_1\leq  \gamma, 
\end{aligned}\]
 where the fourth inequality follows from \eqref{eq:bound:W-1}. 
 Combining the fact that $\bar{x}_k \in \Mcal$, we obtain \eqref{eq:neibohood1}. 
The proof is completed.
\end{proof}

\begin{proof}[Proof of  Lemma \ref{lem:stay-x}]
We prove it by induction on both $\|\hat{z}_k - \bar{z}_k\|$ and $\|\hat{x}_k - \bar{x}_k\|$. Suppose for some $k \geq 0$ such that \eqref{bound-dk-xk-zk} holds. 
       Then we have that
\be
\begin{aligned}
    \|\hat{x}_{k+1} - \bar{x}_{k+1}\| & \leq \| \hat{x}_{k+1} - \bar{z}_k\| \\
    & \leq \| \hat{x}_{k+1} - \hat{z}_k\| + \| \hat{z}_{k} - \bar{z}_k\| \\
    & \leq \frac{1}{n}\sum_{i=1}^n \|x_{i,k+1} - z_{i,k} \| +10 \delta_2\\
    & \overset{\eqref{eq:relation:xk+1-zk}}{\leq} \max_{i} \alpha\| \nabla f_{i}(x_{i,k+1}) - \nabla f_i(x_{i,k}) \| + 10\delta_2 \\
    & \overset{\eqref{eq:bound-nabla-f}}{\leq}  2\delta_2 + 10\delta_2 \leq \delta_1,
\end{aligned}
\ee
where the first inequality use $\bar{x}_k = \Pcal_{\Mcal}(\hat{x}_k)$ and $\bar{z}_k \in \Mcal$.  
In addition, 
\be
\begin{aligned}
    \|\hat{z}_{k+1} - \bar{z}_{k+1} \| & \leq \|\hat{z}_{k+1} - \bar{x}_{k+1} \| \\ 
    & \leq \frac{1}{n}\sum_{i=1}^n \|z_{i,k+1} - \bar{x}_{k+1}\| \\
    & \leq \frac{2}{n}\sum_{i=1}^n \|\sum_{j=1}^n W_{ij}^tx_{j,k} + d_{i,k} - \hat{x}_{k+1}\| \\
    & \leq 2(\sqrt{n} \sigma_2^t  \delta_3 + 4\delta_2) \leq 10\delta_2.
\end{aligned}
\ee
Finally, \eqref{eq:neibohood2} follows from Lemma \ref{lem:stay-neighborhood}. We completed the proof.
\end{proof}

\subsection{Proof of Lemma \ref{lem:descent}}

\begin{proof}
    It follows from the definition of $\varphi_{\alpha}$ in \eqref{eq:DR-env} and $\bar{\by}_{k}\in \mathcal{C}$ that
\be\label{eq:dec-0}
\begin{aligned}
 &  \varphi_{\alpha}^{\rm DR}(\bs_{k})  \\
  \leq & f (\bx_{k+1}) +    \left<\nabla f (\bx_{k+1}), \bar{\by}_{k} - \bx_{k+1}\right> + \frac{1}{2\alpha}\|\bar{\by}_{k} - \bx_{k+1}\|^2 \\
  =  & f (\bx_{k+1}) +   \left<\nabla f (\bx_{k+1}), \bx_k - \bx_{k+1}\right> + \left<\nabla f (\bx_{k+1}), \bar{\by}_k - \bx_k\right> \\
 &+ \frac{1}{2\alpha}\|\bar{\by}_{k} - \bx_{k+1}\|^2 \\
    \leq & f (\bx_{k}) + \frac{L}{2}\|\bx_{k+1} - \bx_k \|^2 +  \left<\nabla f (\bx_{k+1}), \bar{\by}_k - \bx_k\right>\\
    & + \frac{1}{2\alpha}\|\bar{\by}_{k} - \bx_{k+1}\|^2,
\end{aligned}
\ee
    where the second inequality is due to the $L_f$-smoothness of $f_i$'s and $L>L_f$.  Then we have that
    \be\label{eq:dec-1}
\begin{aligned}
  & \varphi_{\alpha}^{\rm DR}(\bs_{k})  
    \leq  f (\bx_{k}) + \frac{L}{2}\|\bx_{k+1} - \bx_k \|^2 + \left<\nabla f (\bx_{k}), \bar{\by}_k - \bx_k\right> \\
   & + \frac{1}{2\alpha}\|\bar{\by}_{k} - \bx_{k+1}\|^2 + \left<\nabla f (\bx_{k+1}) - \nabla f (\bx_{k}), \bar{\by}_k - \bx_k\right> \\
    = &\varphi_{\alpha}^{\rm DR}(\bs_{k-1}) +\left<\nabla f (\bx_{k+1}) - \nabla f (\bx_{k}), \bar{\by}_k - \bx_k\right> \\
    & + \frac{1+\alpha L}{2\alpha}\|\bx_{k+1} - \bx_k \|^2 + \frac{1}{\alpha}\left<\bx_{k+1} - \bx_k, \bx_k - \bar{\by}_k\right>.   
\end{aligned}
\ee
    By \eqref{eq:drs-decen}, we have
    \be \label{eq:est-xy} \begin{aligned}
        \bx_k - \bar{\by}_k & = \bx_k - \bz_k + \bz_k - \bar{\by}_k = \bs_k - \bs_{k+1} + \bz_k - \bar{\by}_k \\
        & = \bx_{k} - \bx_{k+1} + \alpha (\nabla f (\bx_k) - \nabla f (\bx_{k+1})) + \bz_k - \bar{\by}_k.
    \end{aligned} \ee
    Plugging \eqref{eq:est-xy} into \eqref{eq:dec-1} gives
    \be \label{eq:dec-2}
        \begin{aligned}
            &\varphi_{\alpha}^{\rm DR}(\bs_{k-1}) -\varphi_{\alpha}^{\rm DR}(\bs_{k}) \\
        \geq & - \alpha \| \nabla f (\bx_{k+1}) - \nabla f (\bx_k)\|^2  + \frac{1}{2\alpha} \|\bx_{k+1} - \bx_k\|^2 \\
        &- \frac{L}{2} \|\bx_{k+1} - \bx_k\|^2  + \iprod{\nabla f (\bx_{k+1}) - \nabla f (\bx_{k})}{\bz_k - \bar{\by}_k} \\
        &- \frac{1}{\alpha}\iprod{\bx_{k+1} - \bx_k}{\bz_k - \bar{\by}_k} \\
        \geq & \frac{1 - \alpha L - 2\alpha^2 L^2}{2\alpha} \|\bx_{k+1} - \bx_k\|^2 \\
        &- \left(L + \frac{1}{\alpha}\right)\|\bx_{k+1} - \bx_k\| \|\bz_k - \bar{\by}_k\|. 
        \end{aligned}
    \ee
    Summing \eqref{eq:dec-2} over $k$ gives
    \[
    \begin{aligned}
        &\varphi_{\alpha}^{\rm DR}(\bs_0) -\varphi_{\alpha}^{\rm DR}(\bs_{k+1}) 
        \geq   \sum_{\ell=0}^k \frac{1 - \alpha L - 2\alpha^2 L^2}{2\alpha} \|\bx_{\ell +1} - \bx_{\ell}\|^2 \\
        &- \frac{1 + \alpha L}{2 \alpha} \sum_{\ell=0}^k \left( \alpha \| \bx_{\ell +1} - \bx_{\ell} \|^2 + \frac{1}{\alpha} \|\bz_\ell - \bar{\by}_{\ell}\|^2 \right),
    \end{aligned}
    \]
    where the inequality is from the Lipschitz continuity of $\nabla f_i$ and the Cauchy-Schwarz inequality. The proof is completed.
\end{proof}

\subsection{Proof of Theorem \ref{them-main}}

The following lemma shows the discrepancy between $\bz_k$ and $\bar{\by}_k$.
\begin{lemma} \label{lem:z-y-diff}
Suppose that Assumption \ref{assum-w} and \ref{assum:f} hold. Under the conditions as same as in Lemma \ref{lem:stay-x}, it holds that
\be \label{eq:d-mean-diff}
\|\mathbf{d}_{k+1} - (\hat{\bx}_{k+1} - \hat{\bs}_{k+1})\| \leq \sigma_2^t \| \mathbf{d}_{k} - (\hat{\bx}_k - \hat{\bs}_k) \| + \alpha L \| \bx_{k+1} - \bx_k \|
\ee
and 
\be \label{eq:z-y-diff} 
\begin{aligned}
\| \bz_{k+1} - \bar{\by}_{k+1} \| 
\leq & 4 \sigma_2^t (\|\bz_k - \bar{\by}_k \| + \alpha L \| \bx_{k+1} - \bx_k \|) \\
&+ 2 \|\mathbf{d}_{k+1} - (\hat{\bx}_{k+1} - \hat{\bs}_{k+1})\|.
\end{aligned}
\ee
 
\end{lemma}

\begin{proof}
Denote $q_{i,k}: = x_{i,k} - s_{i,k} = \alpha \nabla f(x_{i,k})$.     Since $\hat{\mathbf{d}}_{k+1} = \hat{\bx}_{k+1} - \hat{\bs}_{k+1}$, we have
    \[ \begin{aligned}
       & \| \mathbf{d}_{k+1} - (\hat{\bx}_{k+1} - \hat{\bs}_{k+1}) \| \\
       & \leq \| \mathbf{d}_{k+1} - (\hat{\bx}_k - \hat{\bs}_k) \| \\
        & \leq \| \bW^t \mathbf{d}_k - (\hat{\bx}_k - \hat{\bs}_k) \| + \|\mathbf{q}_{k+1} - \mathbf{q}_k\| \\
        & \leq \sigma_2^t \| \mathbf{d}_k - (\hat{\bx}_k - \hat{\bs}_k) \| + \alpha L \|\bx_{k+1} - \bx_k\|.
    \end{aligned} \]
    Note that
    \be
     \begin{aligned}
        & \|\hat{y}_{k+1} - \bar{y}_{k+1}\| \leq \| \hat{y}_{k+1} - \bar{z}_k  \| \\
         & \leq \|\hat{y}_{k+1} - \hat{z}_k  \| + \| \hat{z}_k - \bar{z}_k \| \\
         & \leq  \frac{1}{n} \sum_{i=1}^n  \|x_{i,k+1} - s_{i,k+1} + x_{i,k+1} - z_{i,k}   \|+ \| \hat{z}_k - \bar{z}_k \| \\
         & \leq   \delta_2 +    2 \delta_2  + 10\delta_2  = 13\delta_2,
     \end{aligned}
    \ee
    which means $\| \hat{y}_{k+1} - \bar{y}_{k+1} \| \leq \gamma$. 
    Note that 
    \be \label{eq:mean}
    \begin{aligned}
    \| \hat{\bz}_k - \bar{\by}_k \|^2 &= n \| \frac{1}{n} \sum_{i=1}^n (z_{i,k} - \bar{y}_k) \|^2 \\
    &\leq n \cdot \frac{1}{n^2} \cdot n \|\bz_k - \bar{\by}_k\|^2 \leq \| \bz_k - \bar{\by}_k\|^2. \end{aligned} \ee
   Then, we have
    \[ \begin{aligned}
        & \|\bz_{k+1} - \bar{\by}_{k+1} \| \\  \leq & 2 \| \bW^t \bx_{k+1} + \mathbf{d}_{k+1} - \hat{\by}_{k+1} \| \\
        \leq & 2 \sigma_2^t \| \bx_{k+1} - \hat{\bx}_{k+1}  \| + 2 \|\mathbf{d}_{k+1} - (\hat{\bx}_{k+1} - \hat{\bs}_{k+1} ) \| \\
        \leq & 2\sigma_2^t (\|\bz_k - \hat{\bz}_k\| + \|\mathbf{q}_{k+1} - \mathbf{q}_k\| + \| \hat{\mathbf{q}}_{k+1} - \hat{\mathbf{q}}_k \| ) \\
        &+ 2 \|\mathbf{d}_{k+1} - (\hat{\bx}_{k+1} - \hat{\bs}_{k+1})\| \\
        \leq & 4 \sigma_2^t (\|\bz_k - \bar{\by}_k\| + \alpha L \|\bx_{k+1} - \bx_k\|)  \\
        &+ 2 \|\mathbf{d}_{k+1} - (\hat{\bx}_{k+1} - \hat{\bs}_{k+1})\|,
    \end{aligned}
    \]
    where the first inequality is from 2-Lipschitz of $\Pcal_{\Mcal}$ over $\bar{U}_{\Mcal}(\gamma)$, the third inequality is due to $\bx_{k+1} = \bz_k + \mathbf{q}_{k+1} - \mathbf{q}_k$ from \eqref{eq:relation:xk+1-zk}, and the last inequality follows from \eqref{eq:mean} and the Lipschitz continuity of $f_i$. We complete the proof.
\end{proof}

With the recursion inequalities \eqref{eq:d-mean-diff} and \eqref{eq:z-y-diff}, we can bound $\sum_{\ell =0}^k \|\bz_{\ell} - \bar{\by}_{\ell}\|^2$ by $\sum_{\ell =0}^k \|\bx_{\ell+1} - \bx_{\ell}\|^2$.

\begin{lemma} \label{lem:z-y-dis}
    Suppose that Assumption \ref{assum-w}, \ref{assum:f} and the conditions in Lemma \ref{lem:stay-x} hold. Then it holds that
    \be \label{eq:z-y-dist-sum} \sum_{\ell =0}^k \|\bz_{\ell} - \bar{\by}_{\ell}\|^2 \leq \mathcal{C}_1\alpha^2 L^2 \sum_{\ell =0}^k \|\bx_{\ell+1} - \bx_{\ell}\|^2 + \mathcal{C}_2,\ee
    where $\mathcal{C}_1,\mathcal{C}_2$ are defined by Theorem \ref{them-main}.
    
\end{lemma}
\begin{proof}
    Note that $\sigma_2^t < 4\sigma_2^t \leq 1/\sqrt{n} \leq 1$. Applying \cite[Lemma 2]{xu2015augmented} to \eqref{eq:d-mean-diff} and \eqref{eq:z-y-diff} gives 
    \be \label{eq:d-sum}
    \begin{aligned}
    \sum_{\ell =0}^k  \|\bd_{k} - (\hat{\bx}_{k} - \hat{\bs}_k) \|^2 \leq  & \frac{4\alpha^2 L^2}{(1-\sigma_2^{t})^2}  \sum_{\ell =0}^k \|\bx_{\ell+1} - \bx_{\ell}\|^2 \\
    &+ \frac{4}{1-\sigma_2^{2t}} \| \bd_{0} - (\hat{\bx}_{0} - \hat{\bs}_0) \|^2. \end{aligned}  \ee
    Similarly, we have that
    \be \label{eq:z-y-sum} 
    \begin{aligned}
    & \sum_{\ell =0}^k \|\bz_{\ell} - \bar{\by}_{\ell} \|^2 \\
    \leq & \frac{32}{(1-4\sigma_2^t)^2} \sum_{\ell =0}^k (4\sigma_2^{2t} \alpha^2 L^2 \|\bx_{\ell+1} - \bx_{\ell}\|^2+ \|\bd_{\ell} - (\hat{\bx}_{\ell} - \hat{\bs}_{\ell})\|^2) \\
    &+ \frac{4}{1-16\sigma_2^{2t}} \| \bz_{0} - \bar{\by}_{0} \|^2 \\
     \leq & \frac{32}{(1-4\sigma_2^t)^2} \sum_{\ell =0}^k (4\sigma_2^{2t} \alpha^2 L^2 +  \frac{4\alpha^2 L^2}{(1-\sigma_2^{t})^2}) \|\bx_{\ell+1} - \bx_{\ell}\|^2 \\
     & + \frac{128}{(1-4\sigma_2^t)^2(1-\sigma_2^{2t})}\| \bd_{0} - (\hat{\bx}_{0} - \hat{\bs}_0) \|^2 \\
    &+ \frac{4}{1-16\sigma_2^{2t}} \| \bz_{0} - \bar{\by}_{0} \|^2.
    \end{aligned} \ee
    The Proof is completed.
\end{proof}

Now we give the proof of Theorem \ref{them-main}.

\begin{proof}[Proof of Theorem \ref{them-main}]
It follows from \eqref{eq:decent} and \eqref{eq:z-y-dist-sum} that
 \be\label{eq:descent1}
    \begin{aligned}
        &\varphi_{\alpha}^{\rm DR}(\bs_0) -\varphi_{\alpha}^{\rm DR}(\bs_{k+1}) \\
        \geq & \frac{(1+\alpha L)(1-\alpha - 2\alpha L - \mathcal{C}_1\alpha L^2 )}{2 \alpha}\sum_{\ell=0}^k \|\bx_{\ell +1} - \bx_{\ell}\|^2 \\
         &-\frac{\mathcal{C}_2(1 + \alpha L)}{2\alpha^2} \\
         \geq & \frac{1}{4\alpha} \sum_{\ell=0}^k \|\bx_{\ell +1} - \bx_{\ell}\|^2 - \frac{\mathcal{C}_2(1 + \alpha L)}{2\alpha^2}, 
    \end{aligned}
    \ee
  where  the first inequality comes from Lemma \ref{lem:z-y-dis}, and the second inequality is due to the assumption on $\alpha$.  By Assumption \ref{assum:f}, we have from \cite[Theorem 3.4]{themelis2020douglas} that $\varphi_{\alpha}^{\rm DR}(\bx_k, \bar{\by}_k) \geq \inf_{\bx} \{ f (\bx) + \delta_{\mathcal{C}}(\bx) \} = f^* > -\infty$. Then, it follows from \eqref{eq:descent1} that
    \be \label{eq:sum-x-bound} 
    \frac{1}{k+1} \sum_{\ell =0}^k \|\bx_{\ell + 1} - \bx_{\ell}\|^2 \leq \frac{4\alpha}{k+1} \left(\varphi_{\alpha}^{\rm DR}(\bx_0, \bar{\by}_0) - f^* + \frac{\mathcal{C}_2}{\alpha^2} \right). \ee
    It follows from the definition of $\bar{\by}_k$ that
    \[ \bar{\by}_{k} - (\hat{\bx}_k - \alpha  \nabla f(\hat{\bx}_k)) \in N_{\bar{\by}_k} \mathcal{M}^n.  \]
    This further implies that
    \be \label{eq:rgrad} \begin{aligned}
        \| \grad f(\bar{\by}_k) \| & = {\rm dist}(\nabla f(\bar{\by}_k), N_{\bar{\by}_k} \Mcal^n) \\ 
        & \leq \| \nabla f(\bar{\by}_k) - \frac{\bar{\by}_k  - \hat{\bx}_k}{\alpha} + \nabla f(\hat{\bx}_k) \| \\
            & \leq \frac{2}{\alpha}\| \bar{\by}_k - \hat{\bx}_k \| \leq \frac{2}{\alpha}\| \bar{\by}_k - \bx_k \|.
    \end{aligned}
    \ee
  By combining with the fact that $\|\bar{\by}_k - \hat{\bx}_k\| \leq \|\bar{\by}_k - \bx_k\|$, we have that 
    \be\label{eq:rgrad1}
    \begin{aligned}
 \| \grad f(\bar{\bx}_k) \| & \leq  \| \grad f(\bar{\by}_k) \| +  \| \grad f(\bar{\by}_k) - \grad f(\bar{\bx}_k) \| \\
 & \leq \frac{2}{\alpha}\| \bar{\by}_k - \bx_k \| + L \|\bar{\by}_k - \bar{\bx}_k\| \\
 & \leq \frac{2}{\alpha}\| \bar{\by}_k - \bx_k \| +   2L\|\bar{\by}_k - \hat{\bx}_k\| \\
  & \leq \frac{2+2\alpha L}{\alpha}\| \bar{\by}_k - \bx_k \| \leq \frac{3}{\alpha}\| \bar{\by}_k - \bx_k \|,
 \end{aligned}
    \ee
    where the third inequality utilizes that 
    $$
    \begin{aligned}
    \|\bar{\by}_k - \bar{\bx}_k\| & \leq \|\bar{\by}_k - \hat{\bx}_k\| + \|\hat{\bx}_k - \bar{\bx}_k\| \\
    & \leq \|\bar{\by}_k - \hat{\bx}_k\| + \|\bar{\by}_k - \hat{\bx}_k\| \leq 2\|\bar{\by}_k - \hat{\bx}_k\|.
    \end{aligned}
    $$
    By the triangle inequality and Proposition \ref{prop}, we have
    \be \label{eq:diff-x-y} 
    \begin{aligned}
       \|\bar{\bx}_k - \bx_k \|  & \leq     \| \bar{\by}_k - \bx_k \| \\
       & \leq \| \bar{\by}_k - \bz_k\| + \|\bz_k - \bx_k \| \\
            & = \| \bar{\by}_k - \bz_k\| + \|\bs_{k+1} - \bs_{k} \| \\
            & \leq \| \bar{\by}_k - \bz_k\| + 2 \|\bx_{k+1} - \bx_k \|,
        \end{aligned}
    \ee
    where the first inequality uses the fact that $\bar{\bx}_k$ belongs to the projection of $\bx_k$ onto $\Mcal^n$, the last inequality follows from \eqref{eq:prox-iso} and $\alpha<1/(2L)$. 
    Then, it holds that 
    \be \label{eq:diff-x-y-sum}
    \begin{aligned}
  &  \frac{1}{k+1}\sum_{\ell=0}^k \|\bar{\bx}_{\ell} - \bx_{\ell} \|^2 \\ \leq & \frac{1}{k+1}\sum_{\ell=0}^k (\|\bar{\by}_\ell - \bz_{\ell}\| + 2 \|\bx_{\ell +1} - \bx_{\ell}\|)^2 \\
     \leq & \frac{2}{k+1} \cdot \left[ (\mathcal{C}_1 \alpha^2L^2 + 4) \sum_{\ell =0}^k \|\bx_{\ell + 1} - \bx_\ell\|^2 + \mathcal{C}_2 \right].
    \end{aligned}
    \ee
    Combining \eqref{eq:rgrad1} and \eqref{eq:diff-x-y-sum} gives
    \be \label{eq:rgrad-sum}  
    \begin{aligned}
    &    \frac{1}{k+1}\sum_{\ell =0}^k \|\grad f(\bar{\bx}_\ell)\|^2  \leq \frac{9}{\alpha^2} \frac{1}{k+1}\sum_{\ell =0}^k \|\bar{\by}_{\ell} - \bx_{\ell}\|^2 \\
        & \underset{\leq}{\eqref{eq:z-y-dist-sum}} \frac{18\mathcal{C}_1 \alpha^2 L^2 + 72}{(k+1) \alpha^2} \sum_{\ell =0}^k \|\bx_{\ell + 1} - \bx_\ell\|^2 + \frac{18 \mathcal{C}_2}{(k+1)\alpha^2}.
    \end{aligned}
    \ee
    Putting \eqref{eq:sum-x-bound}, \eqref{eq:diff-x-y-sum}, and \eqref{eq:rgrad-sum} together gives \eqref{eq:station-consen} and \eqref{eq:station-grad}.
\end{proof}

\section{Convergence of iDDRS}\label{appen-2}

\subsection{Proof of Lemma \ref{lem:stay-x-inex}}

\begin{lemma}\label{lem:bound-sk-xksk-inex}
 Suppose that Assumption \ref{assum-w}-\ref{assum:epsilon} hold.  Let $\{\bs_k,\bx_k,\by_k,\bz_k\}$ be generated by Algorithm \ref{alg:drgta2}. Given any $\delta_2 >0$, if $\epsilon_0 < \delta_2$,  $0 < \alpha \leq \min\{\frac{1}{2L}, \frac{\delta_2 - \epsilon_0}{3\|\nabla f(0)\| + 2L\left(\zeta + \delta_2\right) + 2\epsilon_0} \}$, $t \geq  \left\lceil \log_{\sigma_2}(\frac{1}{4\sqrt{n}}) \right \rceil$, $\|\mathbf{d}_0\|_{F,\infty} \leq 4\delta_2$ and $\|\bs_0\|_{F,\infty} \leq \zeta + \delta_2$, then it holds that for any $k$,
    \begin{align}
    \|\bs_{k}\|_{F,\infty} \leq \zeta + \delta_2, \; \|\bx_k & - \bs_k\|_{F,\infty} \leq \delta_2, \label{eq:bound-sk-index} \\
        \|\alpha \nabla f(x_{i,k}) + \mu_{i,k}\| &\leq  \delta_2, \; i\in [n], \label{eq:bound-nabla-f-index} \\
         \|d_{i,k}\| &\leq 4\delta_2, \;\;\; i\in [n].\label{eq:bound-dk-index}
    \end{align}
\end{lemma}

\begin{proof}
    We prove it by induction on both $\|\bs_k\|_{F,\infty}$ and $\|\bx_k - \bs_k\|_{F,\infty}$. Note that $\|\bs_0\|_{F,\infty} \leq \zeta + \delta_2$ and by Lemma \ref{lem:bound-xs-inex},
    \be \label{bound-xi0-si0-inex}
     \begin{aligned}
        \|x_{i,0} - s_{i,0}\| %& \leq \alpha \left( \|\nabla f(0)\| + \frac{L\|s_{i,0}\| + \epsilon_0 + \|\nabla f(0)\| }{1 - \alpha L} \right) + \epsilon_0 \\ 
        & \leq \alpha \left(3\|\nabla f(0)\| + 2L\left(\zeta + \delta_2 \right) +2 \epsilon_0  \right) + \epsilon_0 \leq \delta_2,
    \end{aligned}
    \ee
    which implies that $\|\bx_0 - \bs_0\|_{F,\infty} \leq \delta_2$. Suppose for some $k\geq 0$ that $\|\bs_k\|_{F,\infty} \leq \zeta + \delta_2$ and $\|\bx_k - \bs_k\|_{F,\infty} \leq \delta_2$.  Then we have 
     \[ \|\bs_{k+1}\|_{F,\infty} \leq \|\bz_k \|_{F,\infty} + \|\bx_k - \bs_k\|_{F,\infty} \leq \zeta + \delta_2. \]
     Similar with \eqref{bound-xi0-si0-inex},  we further have that
     $$
     \begin{aligned}
     &   \|x_{i,k+1} - s_{i,k+1}\| \\
        \leq & \alpha \left(3\|\nabla f(0)\| + 2L\left(\zeta + \delta_2\right) + 2\epsilon_{k+1}  \right) + \epsilon_{k+1} \leq \delta_2,
    \end{aligned}
     $$
 where we use $\epsilon_{k+1} \leq \epsilon_0$.   This implies that $ \|\bx_{k+1} - \bs_{k+1}\| \leq \delta_2$. Moreover, \eqref{eq:bound-nabla-f-index} follows from the fact that $x_{i,k} = s_{i,k} - \alpha \nabla f(x_{i,k}) + \mu_{i,k}$. 
 Finally,  we prove \eqref{eq:bound-dk-index} by induction.  Suppose for some $k\geq 0$ such that \eqref{eq:bound-dk-index} holds.  It follows from $d_{i,0} = x_{i,0} - s_{i,0}$ that
       \be\label{eq:hatd=hatxs-inex}
       \hat{d}_k  = \hat{x}_k - \hat{s}_k = \alpha \hat{g}_k + \hat{\mu}_k,
       \ee
       where $\hat{g}_k: = \frac{1}{n}\sum_{i=1}^n \nabla f(x_{i,k})$. Then we have that
    \begin{equation}
\begin{aligned}
&\| d_{i,k+1} - ( \hat{x}_k - \hat{s}_k )\| \\
 = &  \| \sum_{j=1}^n W_{ij}^td_{j,k} - ( \hat{x}_k - \hat{s}_k ) \| \\
 & + \|\alpha\nabla f_i(x_{i,k+1})+ \mu_{i,k+1}  \| + \| \alpha  \nabla f_i(x_{i,k}) +  \mu_{i,k}  \| \\
 \overset{\eqref{eq:hatd=hatxs-inex}}{\leq} & \| \sum_{j=1}^n (W_{ij}^t - \frac{1}{n})  d_{j,k} \|  + \|\alpha\nabla f_i(x_{i,k+1})+ \mu_{i,k+1}  \| \\
 &+ \| \alpha  \nabla f_i(x_{i,k}) +  \mu_{i,k}  \| \\
 \overset{\eqref{eq:bound:W-1}\eqref{eq:bound-nabla-f-index}}{\leq} & \sigma_2^t \sqrt{n} \max_i \| d_{i,k}\| +2 \delta_2 \\
 \leq & \frac{1}{4}\max_i \| d_{i,k}\| + 2 \delta_2 \leq 3\delta_2.
\end{aligned}
\end{equation}
Hence,
\be
\begin{aligned}
\| d_{i,k+1}   \| & \leq\| d_{i,k+1} - ( \hat{x}_k - \hat{s}_k )\| + \| \hat{x}_k - \hat{s}_k \| \\
&\leq  3\delta_2  +  \max_{i}\|\alpha\nabla f(x_{i,k}) + \mu_{i,k}\|  \overset{\eqref{eq:bound-nabla-f-index}}{\leq} 4\delta_2.
\end{aligned}
\ee  The proof is completed.
\end{proof}

The following lemma is similar to Lemma \ref{lem:stay-neighborhood}, we omit the proof.
\begin{lemma} \label{lem:stay-neighborhood-inex}
 Suppose that Assumption \ref{assum-w}-\ref{assum:epsilon} hold. Let $\{\bs_k,\bx_k,\by_k,\bz_k\}$ be generated by Algorithm \ref{alg:drgta2}.    Under the same condition on $\alpha, \mathbf{d}_0$ and $\bs_0$ as in Lemma \ref{lem:bound-sk-xksk-inex}, if $t\geq \left \lceil \max\{ \log_{\sigma_2}(\frac{1}{4\sqrt{n}}), \log_{\sigma_2}(\frac{\delta_3}{\delta_2 \sqrt{n}}) \} \right \rceil$ and $\bx_k \in \mathcal{N}_1$, then it holds that
\begin{align}
    \sum_{j=1}^n {W_{ij}^t} x_{j,k} +d_{i,k} & \in \bar{U}_{\Mcal} (\gamma), ~ i\in [n], \label{eq:neibohood1-index}
\end{align}
where $\delta_1,\delta_2,\delta_3$ are defined in \eqref{def:delta} and $\mathcal{N}_1$ is defined in \eqref{def:neiborhood}.
\end{lemma}

Now we give the proof of Lemma \ref{lem:stay-x-inex}.

\begin{proof}[Proof of Lemma \ref{lem:stay-x-inex}]

We prove it by induction on both $\|\hat{z}_k - \bar{z}_k\|$ and $\|\hat{x}_k - \bar{x}_k\|$. Suppose for some $k \geq 0$ such that \eqref{bound-dk-xk-zk} holds. It follows from the updated rule of $\bx_k$ and $\bs_k$ in Algorithm \ref{alg:drgta} that
     \be\label{eq:relation:xk+1-zk-inex} \begin{aligned}
         x_{i,k+1} 
         & = z_{i,k} + \alpha \nabla f(x_{i,k})  - \alpha \nabla f(x_{i,k+1}) + \mu_{i,k} - \mu_{i,k+1}, 
     \end{aligned}
       \ee
       Then we have that
\be
\begin{aligned}
    \|\hat{x}_{k+1} - \bar{x}_{k+1}\| & \leq \| \hat{x}_{k+1} - \bar{z}_k\| \\
    & \leq \| \hat{x}_{k+1} - \hat{z}_k\| + \| \hat{z}_{k} - \bar{z}_k\| \\
    & \leq \frac{1}{n}\sum_{i=1}^n \|x_{i,k+1} - z_{i,k} \| +12\delta_2\\
    & \overset{\eqref{eq:bound-nabla-f-index}}{\leq}  2\delta_2  +  10\delta_2 \leq \delta_1,
\end{aligned}
\ee
where the first inequality use $\bar{x}_k = \Pcal_{\Mcal}(\hat{x}_k)$ and $\bar{z}_k \in \Mcal$.  
In addition, 
\be
\begin{aligned}
    \|\hat{z}_{k+1} - \bar{z}_{k+1} \| & \leq \|\hat{z}_{k+1} - \bar{x}_{k+1} \| \\ 
    & \leq \frac{1}{n}\sum_{i=1}^n \|z_{i,k+1} - \bar{x}_{k+1}\| \\
    & \leq \frac{2}{n}\sum_{i=1}^n \|\sum_{j=1}^n W_{ij}^tx_{j,k} + d_{i,k} - \hat{x}_{k+1}\| \\
    & \leq 2(\sqrt{n} \sigma_2^t  \delta_3 + 4\delta_2) \leq 10\delta_2.
\end{aligned}
\ee
Finally, \eqref{eq:neibohood2-index} follows from Lemma \ref{lem:stay-neighborhood-inex}. We completed the proof.

\end{proof}

\subsection{Proof of Theorem \ref{theo:iter:inexact}}

 The following two lemmas are similar to Lemma \ref{lem:z-y-diff} and \ref{lem:z-y-dis}, we omit the proof. 
\begin{lemma} \label{lem:z-y-diff-index}

Suppose that the conditions in Lemma \ref{lem:stay-x-inex} hold. Then, it holds that
\be \label{eq:d-mean-diff-inex}
\begin{aligned}
& \|\mathbf{d}_{k+1} - (\hat{\bx}_{k+1} - \hat{\bs}_{k+1})\| \\ \leq & \sigma_2^t \| \mathbf{d}_{k} - (\hat{\bx}_k - \hat{\bs}_k) \| + \alpha L \| \bx_{k+1} - \bx_k \| + 2\sqrt{n}\epsilon_k
\end{aligned}
\ee
and 
\be \label{eq:z-y-diff-inex} 
\begin{aligned}
&\| \bz_{k+1} - \bar{\by}_{k+1} \| \leq  4 \sigma_2^t (\|\bz_k - \bar{\by}_k \| + \alpha L \| \bx_{k+1} - \bx_k \| \\
&+ 2\sqrt{n}\epsilon_k) + 2 \|\mathbf{d}_{k+1} - (\hat{\bx}_{k+1} - \hat{\bs}_{k+1})\|.
\end{aligned}
\ee
 
\end{lemma}

\begin{lemma} \label{lem:z-y-dis-index}
    Suppose that Assumption \ref{assum-w}-\ref{assum:epsilon} and the conditions in Lemma \ref{lem:stay-x-inex} hold. Then it holds that
    \be \label{eq:z-y-dist-sum-index} \sum_{\ell =0}^k \|\bz_{\ell} - \bar{\by}_{\ell}\|^2 \leq \mathcal{C}_3\alpha^2 L^2 \sum_{\ell =0}^k \|\bx_{\ell+1} - \bx_{\ell}\|^2 + \mathcal{C}_4,\ee
     where $\mathcal{C}_3,\mathcal{C}_4$ and $\mathcal{C}_5$ are defined in Theorem \ref{theo:iter:inexact}. 
\end{lemma}

\begin{proof}[Proof of Theorem \ref{theo:iter:inexact}]
It follows from \eqref{eq:descent-index} and \eqref{eq:z-y-dist-sum-index} that
 \be\label{eq:descent1-index}
    \begin{aligned}
        &\varphi_{\alpha}^{\rm DR}(\bs_0) -\varphi_{\alpha}^{\rm DR}(\bs_{k+1}) \\
        \geq & \frac{(1+\alpha L)(1-\alpha - 2\alpha L -2 \mathcal{C}_3\alpha L^2 )}{2 \alpha}\sum_{\ell=0}^k \|\bx_{\ell +1} - \bx_{\ell}\|^2 \\
         &-\frac{(1 + \alpha L)}{2\alpha^2}(\mathcal{C}_4 + 2\sqrt{n}\sum_{l=0}^k\epsilon_{l})  \\
         \geq & \frac{1}{4\alpha} \sum_{\ell=0}^k \|\bx_{\ell +1} - \bx_{\ell}\|^2-\frac{(1 + \alpha L)}{2\alpha^2}\mathcal{C}_5, 
    \end{aligned}
    \ee
  where  the first inequality comes from Lemma \ref{lem:z-y-dis-index}, and the second inequality is due to the assumption on $\alpha$.  By Assumption \ref{assum:f}, we have from \cite[Theorem 3.4]{themelis2020douglas} that $\varphi_{\alpha}^{\rm DR}(\bx_k, \bar{\by}_k) \geq \inf_{\bx} \{ f (\bx) + \delta_{\mathcal{C}}(\bx) \} = f^* > -\infty$. Then, it follows from \eqref{eq:descent1} that
    \be \label{eq:sum-x-bound-index} 
    \frac{1}{k+1} \sum_{\ell =0}^k \|\bx_{\ell + 1} - \bx_{\ell}\|^2 \leq \frac{4\alpha}{k+1} \left(\varphi_{\alpha}^{\rm DR}(\bx_0, \bar{\by}_0) - f^* + \frac{\mathcal{C}_5}{\alpha^2} \right). \ee
    Similar with \eqref{eq:rgrad}, we have that
    \be \label{eq:rgrad-index} \begin{aligned}
        \| \grad f(\bar{\bx}_k) \| \leq  \frac{3}{\alpha}\| \bar{\by}_k - \bx_k \|. 
    \end{aligned}
    \ee
    By the triangle inequality and Proposition \ref{prop}, we have
    \begin{equation}
    \begin{aligned}
        \|\bs_{k+1} - \bs_{k} \| & \leq \|\bs_{k+1} + \mathbf{\mu}_{k+1} - \bs_{k} - \mathbf{\mu}_{k}\| + \|\mathbf{\mu}_{k+1} - \mathbf{\mu}_{k}\|  \\
        & \leq 2\|\bx_{k+1} - \bx_k \| + 2\sqrt{n}\epsilon_k.
        \end{aligned}
    \end{equation}
    Then we have that
    \be \label{eq:diff-x-y-index} 
    \begin{aligned}
       \|\bar{\bx}_k - \bx_k \|  & \leq     \| \bar{\by}_k - \bx_k \| \\
       & \leq \| \bar{\by}_k - \bz_k\| + \|\bz_k - \bx_k \| \\
            & = \| \bar{\by}_k - \bz_k\| + \|\bs_{k+1} - \bs_{k} \| \\
            & \leq \| \bar{\by}_k - \bz_k\| + 2 \|\bx_{k+1} - \bx_k \|,\\
            & \leq \| \bar{\by}_k - \bz_k\| + 2 \|\bx_{k+1} - \bx_k \|+ 2\sqrt{n}\epsilon_k.
        \end{aligned}
    \ee
    Then, it holds that 
    \be \label{eq:diff-x-y-sum-index}
    \begin{aligned}
  &  \frac{1}{k+1}\sum_{\ell=0}^k \|\bar{\bx}_{\ell} - \bx_{\ell} \|^2 \\ \leq & \frac{1}{k+1}\sum_{\ell=0}^k (\|\bar{\by}_\ell - \bz_{\ell}\| + 2 \|\bx_{\ell +1} - \bx_{\ell}\| + 2\sqrt{n}\epsilon_l)^2 \\
     \leq & \frac{3}{k+1} \cdot \left[ (\mathcal{C}_3 \alpha^2L^2 + 4) \sum_{\ell =0}^k \|\bx_{\ell + 1} - \bx_\ell\|^2 + \mathcal{C}_5 \right].
    \end{aligned}
    \ee
    Combining \eqref{eq:rgrad-index} and \eqref{eq:diff-x-y-sum-index} gives
    \be \label{eq:rgrad-sum-index}  
    \begin{aligned}
    &    \frac{1}{k+1}\sum_{\ell =0}^k \|\grad f(\bar{\by}_\ell)\|^2  \leq \frac{9}{\alpha^2} \frac{1}{k+1}\sum_{\ell =0}^k \|\bar{\by}_{\ell} - \bx_{\ell}\|^2 \\
        & \leq \frac{27\mathcal{C}_1 \alpha^2 L^2 + 108}{(k+1) \alpha^2} \sum_{\ell =0}^k \|\bx_{\ell + 1} - \bx_\ell\|^2 + \frac{27 \mathcal{C}_5}{(k+1)\alpha^2}.\\
     %   & \leq \frac{4(8\mathcal{C}_1 \alpha^2 L^2 + 32)}{(k+1) \alpha} \left(\varphi_{\alpha}^{\rm DR}(\bx_0, \bar{\by}_0) - f^* + \frac{\mathcal{C}_2}{\alpha^2} \right)+ \frac{8 \mathcal{C}_2}{(k+1)\alpha^2}
    \end{aligned}
    \ee
    Putting \eqref{eq:sum-x-bound-index}, \eqref{eq:diff-x-y-sum-index}, and \eqref{eq:rgrad-sum-index} together gives \eqref{eq:station-consen-index} and \eqref{eq:station-grad-index}.
\end{proof}

\bibliographystyle{IEEEtran}
\bibliography{ref}
\end{document}